\documentclass{article}
\usepackage[utf8]{inputenc}

\usepackage{vmargin}

\usepackage{amsmath,amssymb,amsthm}
\usepackage[usenames,dvipsnames]{color}
\usepackage{stmaryrd}
\usepackage{enumerate}
\usepackage[algoruled,vlined,english,linesnumbered]{algorithm2e}
\usepackage[pdfpagelabels,colorlinks=true,citecolor=blue]{hyperref}
\usepackage{comment}
\usepackage{tikz}
\usepackage{bbm}
\usepackage{multirow}

\definecolor{beige}{rgb}{0.9,0.9,0.7}

\newcommand{\noopsort}[1]{}

\DeclareMathOperator{\val}{val}

\begin{document}

\newtheorem{theo}{Theorem}[section]
\newtheorem{lem}[theo]{Lemma}
\newtheorem{prop}[theo]{Proposition}
\newtheorem{cor}[theo]{Corollary}
\newtheorem{quest}[theo]{Question}
\theoremstyle{definition}
\newtheorem{rem}[theo]{Remark}
\newtheorem{ex}[theo]{Example}
\newtheorem{deftn}[theo]{Definition}
\newtheorem{rmk}[theo]{Remark}

\newcommand{\N}{\mathbb N}
\newcommand{\Z}{\mathbb Z}
\newcommand{\Zp}{\Z_p}
\newcommand{\Q}{\mathbb Q}
\newcommand{\Qp}{\Q_p}
\newcommand{\F}{\mathbb F}
\newcommand{\Fp}{\F_p}
\newcommand{\R}{\mathbb R}
\renewcommand{\O}{\mathcal O}
\newcommand{\m}{\mathfrak m}
\newcommand{\calL}{\mathcal L}

\newcommand{\M}{\text{\tt M}}

\renewcommand{\min}{\text{\rm min}}
\renewcommand{\max}{\text{\rm max}}

\newcommand{\ring}{\mathfrak A}
\newcommand{\fracring}{\mathfrak K}
\renewcommand{\prec}{\text{\rm prec}}
\renewcommand{\val}{\text{\rm val}}
\newcommand{\id}{\text{\rm id}}
\renewcommand{\Res}{\text{\rm Res}}
\newcommand{\lc}{\text{\rm lc}}
\newcommand{\Card}{\text{\rm Card}\:}

\renewcommand{\P}{\mathbb P}
\newcommand{\E}{\mathbb E}
\newcommand{\Var}{\text{\rm Var}}
\newcommand{\Cov}{\text{\rm Cov}}

\newcommand{\lb}{\ensuremath{\llbracket}}
\newcommand{\rb}{\ensuremath{\rrbracket}}
\newcommand{\lp}{(\!(}
\newcommand{\rp}{)\!)}
\newcommand{\col}{\: : \:}

\newcommand{\A}{W}

\def\todo#1{\ \!\!{\color{red} #1}}
\definecolor{purple}{rgb}{0.6,0,0.6}

\def\binom#1#2{\Big(\begin{array}{cc} #1 \\ #2 \end{array}\Big)}

\title{Resultants and subresultants of $p$-adic polynomials}

\author{Xavier Caruso}

\maketitle

\begin{abstract}
We address the problem of the stability of the computations of resultants 
and subresultants of polynomials defined over complete discrete valuation 
rings (\emph{e.g.} $\Zp$ or $k[[t]]$ where $k$ is a field). We prove that 
Euclide-like algorithms are highly unstable on average and we explain, in
many cases, how one can stabilize them without sacrifying the complexity.
On the way, we completely determine the distribution of the valuation of
the subresultants of two random monic $p$-adic polynomials having the
same degree.
\end{abstract}

\section{Introduction}

As wonderfully illustrated by the success of Kedlaya-type counting 
points algorithms \cite{kedlaya}, $p$-adic technics are gaining nowadays 
more and more popularity in computer science, and more specifically in 
symbolic computation. A crucial issue when dealing with $p$-adics is 
those of stability. Indeed, just like real numbers, $p$-adic numbers are 
by nature infinite and thus need to be truncated in order to fit in the 
memory of a computer. The level of truncation is called the 
\emph{precision}. Usual softwares implementing $p$-adics (\emph{e.g.} 
\textsc{magma} \cite{magma}, \textsc{pari} \cite{pari}, \textsc{sage} 
\cite{sage}) generally tracks the precision as follows: an individual 
precision is attached to any $p$-adic variable and this precision is 
updated after each basic arithmetic operation. This way of tracking 
precision can be seen as the analogue of the arithmetic intervals in the 
real setting. We refer to \S \ref{subsubsec:computDVR} for more details.

In the paper \cite{padicprec}, the authors propose a new framework to 
control $p$-adic precision. The aim of this paper is to illustrate the 
technics of \emph{loc. cit.} on the concrete example of computation of 
\textsc{gcd}s and subresultants of $p$-adic polynomials. There is 
actually a real need to do this due to the combination of two reasons: 
on the one hand, computating \textsc{gcd}s is a very basic operation for 
which it cannot be acceptable to have important instability whereas, on 
the other hand, easy experimentations show that all standard algorithms 
for this task (\emph{e.g.} extended Euclide's algorithm) are \emph{very} 
unstable.
\begin{figure}
\renewcommand{\arraystretch}{1.2}
\begin{center}
\begin{tabular}{|c||c|c|}
\hline
\multirow{2}{*}{\hspace{0.2cm}Degree\hspace{0.2cm}\null} & 
\multicolumn{2}{c|}{\hspace{0.1cm}Loss of precision (in number of significant digits)\hspace{0.1cm}\null}\\
\cline{2-3}
& \hspace{0.5cm}Euclide algorithm\hspace{0.5cm}\null & expected \\
\hline
$5$ & $\phantom{00}6.3$ & $3.1$ \\
$10$ & $\phantom{0}14.3$ & $3.2$ \\
$25$ & $\phantom{0}38.9$ & $3.2$ \\
$50$ & $\phantom{0}79.9$ & $3.2$ \\
$100$ & $160.0$ & $3.2$ \\
\hline
\end{tabular}
\end{center}

\vspace{-0.3cm}

\caption{Average loss of precision when computing the {\sc gcd}
of two random monic polynomial of fixed degree over $\Z_2$.}
\label{fig:precision}

\vspace{-0.3cm}
\end{figure}
Figure \ref{fig:precision} illustrates the instability of the classical
extended Euclide's algorithm (\emph{cf} Algorithm \ref{algo:Euclide})
when it is called on random inputs which are monic $2$-adic 
polynomials of fixed degree (see also Example \ref{ex:Euclide}). 
\begin{algorithm}
  \SetKwInOut{Input}{Input}
  \SetKwInOut{Output}{Output}
  \Input{Two polynomials $A, B \in \Qp[X]$ (whose coefficients are known at given precision)}
  \Output{A triple $D, U, V$ such that $D = AU + BV = \gcd(A,B)$}

  \BlankLine

  $S_1 \leftarrow A$; $U_1 \leftarrow 1$; $V_1 \leftarrow 0$

  $S_2 \leftarrow B$; $U_2 \leftarrow 0$; $V_2 \leftarrow 1$

  $k \leftarrow 2$

  \While{$S_k \neq 0$}
    {$Q, S_{k+1} \leftarrow$ quotient and remainder in the Euclidean
     division of $S_{k-1}$ by $S_k$

     $U_{k+1} \leftarrow U_{k-1} - Q U_k$

     $V_{k+1} \leftarrow V_{k-1} - Q V_k$

     $k \leftarrow k+1$}

  \Return{$S_{k-1}, U_{k-1}, V_{k-1}$}
\caption{Extended Euclide's algorithm}
\label{algo:Euclide}
\end{algorithm}
Looking at the last line, we see that 
extended Euclide's algorithm outputs the B\'ezout coefficients of two 
monic $2$-adic polynomials of degree $100$ with an average loss of 
$160$ significant digits by coefficient whereas a stable 
algorithm should only loose $3.2$ digits on average. This 
``theoretical'' loss is computed as the double of the valuation of the 
resultant. Indeed Cramer-like formulae imply that B\'ezout coefficients 
can be computed by performing a unique division by the resultant, 
inducing then only the aforementioned loss of precision (see 
\S \ref{subsubsec:computDVR}, Eq.~\eqref{eq:precdiv} for a full
justification). Examining the 
table a bit more, we observe that the ``practical'' loss of 
precision due to Euclide's algorithm seems to grow linearly with respect 
to the degree of the input polynomials whereas the ``theoretical'' loss 
seems to be independant of it. In other words, the instability of 
Euclide's algorithm is becoming more and more critical when the degree 
of the input increases.

\paragraph{Content of the paper}

The aim of this article is twofold. We first provide in \S 
\ref{sec:unstable} a theoretical study of the instability phenomenon 
described above and give strong evidences that the loss of precision 
grows linearly with respect to the degree of the input polynomials, as we 
observed empirically. In doing so, we determine the distribution of the 
valuation of the subresultants of random monic polynomials over $\Zp$ 
(\emph{cf} Theorem \ref{th:lawVj}). This is an independant result which 
has its own interest.

Our second goal, which is carried out in \S \ref{sec:stable}, is to rub 
out these unexpected losses of precision. Making slight changes to the 
standard subresultant pseudo-remainder sequence algorithm and using in 
an essential way the results of \cite{padicprec}, we manage to design a stable 
algorithm for computing all subresultants of two monic polynomials over 
$\Zp$ (satisfying an additional assumption). This basically allows 
to stably compute \textsc{gcd}s assuming that the degree of the 
\textsc{gcd} is known in advance.

\paragraph{Notations}

Figure~\ref{fig:notations} summerizes the main notations used in
this paper. The definitions of many of them will be recalled in 
\S \ref{subsec:setting}.

\begin{figure}
\begin{center}
\begin{tabular}{rcl}
$\ring$ &--- & a commutative ring (without any further assumption) \\
$\A$ &---& a complete discrete valuation ring \\
$\pi$ &---& a uniformizer of $\A$ \\
$K$ &---& the fraction field of $\A$ \\
$k$ &---& the residue field of $\A$ \smallskip \\

$\ring_{<n}[X]$ &---& the free $\ring$-module consisting of
polynomials over $\ring$ of degree $<n$ \\
$\ring_{\leq n}[X]$ &---& the free $\ring$-module consisting of
polynomials over $\ring$ of degree $\leq n$ \\
$\ring_n[X]$ &---& the affine space consisting of \emph{monic}
polynomials over $\ring$ of degree $n$. \smallskip \\

$\Res^{d_A,d_B}(A,B)$ &---& The resultant of $A$ and
$B$ ``computed in degree $(d_A,d_B)$'' \\
$\Res_j^{d_A,d_B}(A,B)$ &---& The $j$-th subresultant of $A$ and
$B$ ``computed in degree $(d_A,d_B)$'' \\
\end{tabular}
\end{center}

\vspace{-0.2cm}

\caption{Notations used in the paper}
\label{fig:notations}
\end{figure}

\section{The setting}
\label{subsec:setting}

The aim of this section is to introduce the setting we shall work in 
throughout this paper (which is a bit more general than those considered 
in the introduction).

\subsection{Complete discrete valuation rings}

\begin{deftn}
A \emph{discrete valuation ring} (DVR for short) is a domain $\A$
equipped with a map $\val : \A \to \Z \cup \{+\infty\}$ --- the 
so-called \emph{valuation} --- satisfying the four axioms: 
\begin{enumerate}
\item $\val(x) = +\infty$ iff $x = 0$
\item $\val(xy) = \val(x) + \val(y)$
\item $\val(x+y) \geq \min(\val(x), \val(y))$
\item any element of valuation $0$ is invertible.
\end{enumerate}
\end{deftn}

Throughout this paper, we fix a discrete valuation ring $\A$ and assume 
that the valuation on it is normalized so that it takes the value $1$. 
We recall that $\A$ admits a unique maximal ideal $\m$, consisting of 
elements of positive valuation. This ideal is principal and 
generated by any element of valuation $1$. Such an element is called a 
\emph{uniformizer}. Let us fix one of them and denote it by $\pi$. The 
\emph{residue field} of $\A$ is the quotient $\A/\m = \A/\pi\A$ and we 
shall denote it by $k$.

The valuation defines a distance $d$ on $\A$ by letting $d(x,y) = 
e^{-\val(x-y)}$ for all $x, y \in \A$. We say that $\A$ is 
\emph{complete} if it is complete with respect to $d$, in the sense that 
every Cauchy sequence converges. Assuming that $\A$ is complete, any 
element $x \in \A$ can be written uniquely as a convergent series:
\begin{equation}
\label{eq:expandCDVR}
x = x_0 + x_1 \pi + x_2 \pi^2 + \cdots + x_n \pi^n + \cdots
\end{equation}
where the $x_i$'s lie in a fixed set $S$ of representatives of classes
modulo $\pi$. Therefore, as an additive group, $\A$ is isomorphic to 
the set of sequences $\N \to k$. On the contrary, the multiplicative
structure may vary.

Let $K$ denote the fraction field of $\A$. The valuation $v$ extends 
uniquely to $K$ by letting $\val(\frac x y) = \val(x) - \val(y)$. 
Moreover, it follows from axiom 4 that $K$ is obtained from $\A$ by 
inverting $\pi$. Thus, any element of $K$ can be uniquely written as an 
infinite sum:
\begin{equation}
\label{eq:expandCDVF}
x = \sum_{i=i_0}^\infty x_i \pi^i
\end{equation}
where $i_0$ is some relative integer and the $x_i$'s are as above. The 
valuation of $x$ can be easily read off this writing: it is the smallest 
integer $i$ such that $x_i \not\equiv 0 \pmod \pi$.

\subsubsection{Examples}

A first class of examples of discrete valuation rings 
are rings of formal power series over a field. They are equipped with 
the standard valuation defined as follows: $\val(\sum_{i \geq 0} a_i 
t^i)$ is the smallest integer $i$ with $a_i \neq 0$. The corresponding 
distance on $k[[t]]$ is complete. Indeed, denoting by $f[i]$ the term in 
$t^i$ in a series $f \in k[[t]]$, we observe that a sequence $(f_n)_{n 
\geq 0}$ is Cauchy if and only if the sequences $(f_n[i])_{n \geq 0}$ 
are all ultimately constant. A Cauchy sequence $(f_n)_{n \geq 0}$ 
therefore converges to $\sum_{i \geq 0} a_i t^i$ where $a_i$ is the 
limit of $f_n[i]$ when $n$ goes to $+\infty$.
The DVR $k[[t]]$ has a distinguished uniformizer, namely $t$. Its 
maximal ideal is then the principal ideal $(t)$ and its residue field
is canonically isomorphic to $k$. 
If one chooses $\pi = t$ and constant polynomials as representatives of 
classes modulo $t$, the expansion \eqref{eq:expandCDVR} is nothing but 
the standard writing of a formal series.
The fraction field of $k[[t]]$ is the ring of Laurent series over $k$
and, once again, the expansion \eqref{eq:expandCDVF} corresponds to the
usual writing of Laurent series.

\medskip

The above example is quite important because it models all complete 
discrete valuation rings of equal characteristic, \emph{i.e.} whose 
fraction field and residue field have the same characteristic. On the 
contrary, in the mixed characteristic case (\emph{i.e.} when the 
fraction field has characteristic $0$ and the residue field has positive 
characteristic), the picture is not that simple.
Nevertheless, one can construct several examples and, among them, the 
most important is certainly the ring of $p$-adic integers $\Zp$ (where 
$p$ is a fixed prime number). It is defined as the projective limit of 
the finite rings $\Z/p^n\Z$ for $n$ varying in $\N$. In concrete terms, 
an element of $\Zp$ is a sequence $(x_n)_{n \geq 0}$ with $x_n \in 
\Z/p^n\Z$ and $x_{n+1} \equiv x_n \pmod{p^n}$. The addition (resp. 
multiplication) on $\Zp$ is the usual coordinate-wise addition (resp. 
multiplication) on the sequences. The $p$-adic valuation of $(x_n)_{n 
\geq 0}$ as above is defined as the smallest integer $i$ such that $x_i 
\neq 0$. We can easily check that $\Zp$ equipped with the $p$-adic 
valuation satisfies the four above axioms and hence is a DVR. A
uniformizer of $\Zp$ is $p$ and its residue field is $\Z/p\Z$. A
canonical set of representatives of classes modulo $p$ is $\{0, 1,
\ldots, p-1\}$.

Given a $p$-adic integer $x = (x_n)_{n \geq 0}$, the $i$-th digit of 
$x_n$ in $p$-basis is well defined as soon as $i<n$ and the 
compatibility condition $x_{n+1} \equiv x_n \pmod{p^n}$ implies that 
it does not depend on $n$. As a consequence, a $p$-adic integer can 
alternatively be represented as a ``number'' written in $p$-basis having 
an infinite number of digits, that is a formal sum of the shape:
\begin{equation}
\label{eq:expandZp}
a_0 + a_1 p + a_2 p^2 + \cdots + a_n p^n + \cdots
\quad \text{with } a_i \in \{0, 1, \ldots, p-1\}.
\end{equation}
Additions and multiplications can be performed on the above writing
according to the rules we all studied at school (and therefore taking
care of carries). Similarly to the equal characteristic case, we prove
that $\Zp$ is complete with respect to the distance associated to the 
$p$-adic valuation. The writing \eqref{eq:expandZp} corresponds to the 
expansion \eqref{eq:expandCDVR} provided that we have chosen $\pi = p$ 
and $S = \{0, 1, \ldots, p-1\}$.
The fraction field of $\Zp$ is denoted by $\Qp$.

\subsubsection{Symbolic computations over DVR}
\label{subsubsec:computDVR}

We now go back to a general complete discrete valuation ring $\A$, 
whose fraction field is still denoted by $K$.
The memory of a computer being necessarily finite, it is not possible to 
represent exhaustively all elements of $\A$. Very often, mimicing 
what we do for real numbers, we choose to truncate the expansion 
\eqref{eq:expandCDVR} at some finite level. Concretely, this means 
that we work with approximations of elements of $\A$ of the form
\begin{equation}
\label{eq:approx}
x = \sum_{i=0}^{N-1} x_i \pi^i + O(\pi^N) \quad \text{with } N \in \N
\end{equation}
where the notation $O(\pi^N)$ means that the $x_i$'s with $i \geq N$ are 
not specified.

\begin{rem}
\label{rem:balls}
From a theoretical point of view, the expression~\eqref{eq:approx}
does not represent a single element $x$ of $\A$ but an open ball in
$\A$, namely the ball of radius $e^{-N}$ centered at $\sum_{i=0}^{N-1} 
x_i \pi^i$ (or actually any element congruent to it modulo $\pi^N$).
In other words, on a computer, we cannot work with actual $p$-adic
numbers and we replace them by balls which are more tractable (at
least, they can be encoded by a finite amount of information).
\end{rem}

The integer $N$ appearing in Eq.~\eqref{eq:approx} is the so-called 
\emph{absolute precision} of $x$. The \emph{relative precision} of $x$ 
is defined as the difference $N-v$ where $v$ denotes the valuation of 
$x$. Continuing the comparison with real numbers, the relative precision 
corresponds to the number of significant digits since $x$ can be 
alternatively written:
$$x = p^v \sum_{j=0}^{N-v-1} y_j \pi^j + O(\pi^N)
\quad \text{with } y_j = x_{j+v} \text{ and } y_0 \neq 0.$$
Of course, it may happen that all the $x_i$'s ($0 \leq i < N$) vanish,
in which case the valuation of $x$ is undetermined. In this particular
case, the relative precision of $x$ is undefined.

There exist simple formulas to following precision after each single
elementary computation. For instance, basic arithmetic operations can
be handled using:
\begin{align}
\big(a + O(\pi^{N_a})\big) + \big(b + O(\pi^{N_b})\big) 
& = a + b + O(\pi^{\min(N_a,N_b)}), 
\label{eq:precadd} \\
\big(a + O(\pi^{N_a})\big) - \big(b + O(\pi^{N_b})\big) 
& = a - b + O(\pi^{\min(N_a,N_b)}), 
\label{eq:precsub} \\
\big(a + O(\pi^{N_a})\big) \times \big(b + O(\pi^{N_b})\big) 
& = ab + O(\pi^{\min(N_a + \val(b), N_b + \val(a))}).
\label{eq:precmul} \\
\big(a + O(\pi^{N_a})\big) \div \big(b + O(\pi^{N_b})\big) 
& = \frac a b + O(\pi^{\min(N_a - \val(b), N_b + \val(a) - 2 \val(b))}).\footnotemark
\label{eq:precdiv} 
\end{align}

\footnotetext{We observe that these formulas can be rephrased as follows: 
the absolute (resp. relative) precision on the result of a sum or a
substraction (resp. a product or a division) is the minimum of the 
absolute (resp. relative) precisions on .}

\noindent
with the convention that $\val(a) = N_a$ (resp. $\val(b) = N_b$) if all
known digits of $a$ (resp. $b$) are zero.
Combining these formulas, one can track the precision while executing 
any given algorithm. This is the analogue of the standard \emph{interval 
arithmetic} over the reals. Many usual softwares (as \textsc{sage}, 
\textsc{magma}) implement $p$-adic numbers and formal series this way. 
We shall see later that this often results in overestimating the losses 
of precision.

\begin{ex}
\label{ex:Euclide}
As an illustration, let us examine the behaviour of the precision
on the sequence $(R_i)$ while executing Algorithm~\ref{algo:Euclide} 
with the input:
\begin{align*}
A & = X^5 + 
      \big(27 + O(2^5)\big) X^4 +
      \big(11 + O(2^5)\big) X^3 +
      \big(5 + O(2^5)\big) X^2 +
      \big(18 + O(2^5)\big) X +
      \big(25 + O(2^5)\big) \\
B & = X^5 + 
      \big(24 + O(2^5)\big) X^4 +
      \big(25 + O(2^5)\big) X^3 +
      \big(12 + O(2^5)\big) X^2 +
      \big(3 + O(2^5)\big) X +
      \big(10 + O(2^5)\big).
\end{align*}
The remainder in the Euclidean division of $A$ by $B$ is $S_3 = A-B$.
According to Eq.~\eqref{eq:precsub},
we do not loose precision while performing this substraction and the
result we get is:
$$S_3 = 
      \big(3 + O(2^5)\big) X^4 +
      \big(18 + O(2^5)\big) X^3 +
      \big(25 + O(2^5)\big) X^2 +
      \big(15 + O(2^5)\big) X +
      \big(15 + O(2^5)\big).$$ 
In order to compute $S_4$, we have now to perform the Euclidean division 
of $S_2 = B$ by $S_3$. Noting that the leading coefficient of $S_2$ has 
valuation $0$ and using Eq~\eqref{eq:precadd}--\eqref{eq:precdiv}, we 
deduce that this operation does not loose precision again. We get:
$$S_4 = 
      \big(26 + O(2^5)\big) X^3 +
      \big(17 + O(2^5)\big) X^2 +
      \big(4 + O(2^5)\big) X +
      \big(16 + O(2^5)\big).$$
We observe now that the leading coefficient of $S_4$ has valuation $1$. 
According to Eq.~\eqref{eq:precdiv}, divising by this coefficient ---
and therefore \emph{a fortioti} computing the euclidean division of 
$S_3$ by $S_4$ --- will result in loosing at least one digit in
relative precision. The result we find is:
$$S_5 = 
      \underbrace{\big(\textstyle \frac 3 4 + O(2^2)\big)}_{\text{rel. prec.} = 4} X^2 +
      \underbrace{\big(6 + O(2^3)\big)}_{\text{rel. prec.} = 2} X +
      \underbrace{\big(3 + O(2^3)\big)}_{\text{rel. prec.} = 3}.$$
Continuing this process, we obtain:
$$S_6 = 
      \big(20 + O(2^5)\big) X +
      \big(12 + O(2^5)\big) 
\quad \text{and} \quad
S_7 =
      \textstyle \frac 7 4 + O(2).$$
The relative precision on the final result $S_7$ is then $3$, which is
less than the initial precision which was $5$.
\end{ex}

\subsection{Subresultants}
\label{subsec:subres}

A first issue when dealing with numerical computations of \textsc{gcd}s 
of polynomials over $\A$ is that the \textsc{gcd} function is not 
continuous: it takes the value $1$ on an open dense subset without being 
constant. This of course annihilates any hope of computing \textsc{gcd}s 
of polynomials when only approximations of them are known. Fortunately, 
there exists a standard way to recover continuity in this context: it 
consists in replacing \textsc{gcd}s by subresultants which 
are playing an analoguous role. For this reason, in what follows, we will 
exclusively consider the problem of computing subresultants.

\subsubsection*{Definitions and notations}

We recall briefly basic definitions and results about resultants and 
subresultants. For a more complete exposition, we refer to \cite[\S 
4.2]{real}, \cite[\S 3.3]{cohen} and \cite[\S 4.1]{winkler}. Let $\ring$ 
be an arbitrary ring and let $A$ and $B$ be two polynomials with 
coefficients in $\ring$. We pick in addition two integers $d_A$ and 
$d_B$ greater than or equal to the degree of $A$ and $B$ respectively. 
We consider the Sylvester application:
$$\begin{array}{rcl}
\psi : \quad
\ring_{< d_B}[X] \times \ring_{< d_A}[X] & \to &
\ring_{< d_A+d_B}[X] \smallskip \\
(U,V) & \mapsto & AU + BV
\end{array}$$
where $\ring_{<d}[X]$ refers to the finite free $\ring$-module of
rank $d$ consisting of polynomials over $\ring$ of degree strictly 
less than $d$.
The Sylvester matrix is the matrix of $\psi$ in the canonical ordered
basis, which are
$$\begin{array}{rl}
((X^{d_B-1},0), \ldots, (X,0), (1,0), (0,X^{d_A-1}), \ldots, (0,1)) &
\text{for the source} \smallskip \\
\text{and} \quad (X^{d_A+d_B-1}, \ldots, X, 1) &
\text{for the target.}
\end{array}$$
The \emph{resultant} of $A$ and $B$ (computed in degree $d_A, d_B$) is the 
determinant of the $\psi$; we denote it by $\Res^{d_A,d_B}(A,B)$. We 
observe that it vanishes if $d_A > \deg A$ or $d_B > \deg B$. In what 
follows, we will freely drop the exponent $d_A, d_B$ if $d_A$ and $d_B$ 
are the degrees of $A$ and $B$ respectively.
Using Cramer formulae, we can build polynomials $U^{d_A, d_B}(A,B) \in
\ring_{< d_B}[X]$ and $V^{d_A, d_B}(A,B) \in \ring_{< d_A}[X]$ 
satisfying the two following conditions:
\begin{enumerate}[i)]
\item their coefficients are, up to a sing, maximal minors of the 
Sylvester matrix, and
\item $A \cdot U^{d_A, d_B}(A,B) + B \cdot V^{d_A, d_B}(A,B) = 
\Res^{d_A, d_B}(A,B)$.
\end{enumerate}
These polynomials are called the \emph{cofactors} of $A$ and $B$
(computed in degree $d_A, d_B$).

\medskip

The subresultants are defined in the similar fashion. Given an
integer $j$ in the range $[0, d)$ where $d = \min(d_A, d_B)$, we
consider the ``truncated'' Sylvester application:
$$\begin{array}{rcl}
\psi_j : \quad
\ring_{< d_B-j}[X] \times \ring_{< d_A-j}[X] & \to &
\ring_{< d_A+d_B-j}[X]/\ring_{< j}[X] \smallskip \\
(U,V) & \mapsto & AU + BV.
\end{array}$$
Its determinant (in the canonical basis) is the $j$-th \emph{principal 
subresultant} of $A$ and $B$ (computed in degree $d_A, d_B$). Just 
as before, we can construct polynomials $U_j^{d_A, d_B}(A,B) \in
\ring_{< d_B-j}[X]$ and $V_j^{d_A, d_B}(A,B) \in \ring_{< d_A-j}[X]$
such that:
\begin{enumerate}[i)]
\item their coefficients are, up to a sing, maximal minors of the 
Sylvester matrix\footnote{Indeed, observe that the matrix of $\psi_j$ is a
submatrix of the Sylvester matrix.}, and
\item $A \cdot U_j^{d_A, d_B}(A,B) + B \cdot V_j^{d_A, d_B}(A,B) 
\equiv \det \psi_j \pmod {\ring_{< j}[X]}$.
\end{enumerate}
We set $R_j^{d_A, d_B}(A,B) = A \cdot U_j^{d_A, d_B}(A,B) + B \cdot 
V_j^{d_A, d_B}(A,B)$: it is the $j$-th \emph{subresultant} of $A$
and $B$ (computed in degree $d_A, d_B$). The above congruence implies
that $R_j^{d_A, d_B}(A,B)$ has degree at most $j$ and that its
coefficient of degree $j$ is the $j$-th principal subresultant of $A$
and $B$. As before, we freely drop the exponent $d_A, d_B$ when
$d_A$ and $d_B$ are equal to the degrees of $A$ and $B$ respectively.
When $j=0$, the application $\psi_j$ is nothing but $\psi$. Therefore,
$\Res_0^{d_A, d_B}(A,B) = \Res^{d_A, d_B}(A,B)$ and, similarly, the
cofactors agree: we have $U_0^{d_A, d_B}(A,B) = U^{d_A, d_B}(A,B)$
and $V_0^{d_A, d_B}(A,B) = V^{d_A, d_B}(A,B)$.

We recall the following very classical result.

\begin{theo}
We assume that $\ring$ is a field. Let $A$ and $B$ be two polynomials
with coefficients in $\ring$. Let $j$ be the smallest integer such
that $\Res_j(A,B)$ does not vanish. Then $\Res_j(A,B)$ is a \textsc{gcd}
of $A$ and $B$.
\end{theo}

Since they are defined as determinants, subresultants behave well with 
respect to base change: if $f : \ring \to \ring'$ is a morphism of rings 
and $A$ and $B$ are polynomials over $\ring$ then $\Res_j^{d_A,d_B} 
(f(A),f(B)) = f\big(\Res_j^{d_A,d_B}(A,B)\big)$ where $f(A)$ and $f(B)$ 
denotes the polynomials deduced from $A$ and $B$ respectively by 
applying $f$ coefficient-wise. This property is sometimes referred to as 
the \emph{functoriality} of subresultants.
We emphasize that, when $f$ is not injective, the relation 
$\Res_j(f(A),f(B)) = f\big(\Res_j(A,B)\big)$ does \emph{not} hold in 
general since applying $f$ may decrease the degree. Nevertheless, if
$d_A$ and $d_B$ remained fixed, this issue cannot happen.

\subsubsection*{The subresultant pseudo-remainder sequence}

When $\ring$ is a domain, there exists a standard nice Euclide-like
reinterpreation of subresultants, which provides in particular an 
efficient algorithm for computing them. Since it will play an important 
role in this paper, we take a few lines to recall it.

This reinterpretation is based on the so-called \emph{subresultant 
pseudo-remainder sequence} which is defined as follows. We pick $A$ and 
$B$ as above. Denoting by $(P \,\%\, Q)$ 
the remainder in the Euclidean division of $P$ by $Q$, we define two 
recursive sequences $(S_i)$ and $(c_i)$ as follows:
\begin{equation}
\label{eq:subressequence}
\left\{\begin{array}{ll}
S_{-1} = A, \, S_0 = B, \, c_{-1} = 1 \smallskip \\
\displaystyle S_{i+1} = (-s_i)^{\varepsilon_i + 1}
s_{i-1}^{-1} \: c_i^{-\varepsilon_i} \cdot (S_{i-1} \,\%\, S_i)
& \text{for } i \geq 0 \smallskip \\
\displaystyle c_{i+1} = s_{i+1}^{\varepsilon_{i+1}} \cdot c_i^{1-\varepsilon_{i+1}}
& \text{for } i \geq -1. 
\end{array}\right.
\end{equation}
Here $n_i = \deg S_i$, $\varepsilon_i = n_{i+1} - n_i$ and $s_i$ is the 
leading coefficient of $S_i$ if $i \geq 0$ and $s_{-1} = 1$ by 
convention. These sequences are finite and the above recurrence applies 
until $S_i$ has reached the value $0$.

\begin{prop}
\label{prop:prem}
With the above notations, we have:
$$\begin{array}{rcll}
\Res_j(A,B) & = & S_i & \text{if } j = n_{i-1} - 1 \smallskip \\
& = & 0 & \text{if } n_i < j < n_{i-1} - 1 \\
& = & \!\big(\frac{s_i}{s_{i-1}}\big)^{\varepsilon_i-1} \cdot S_i & 
\text{if } j = n_i
\end{array}$$
for all $i$ such that $S_i$ is defined.
\end{prop}

\begin{rem}
The Proposition \ref{prop:prem} provides a formula for \emph{all} 
subresultants. We note moreover that, in the common case where $n_{i-1} 
= n_i - 1$, the two formulas giving $\Res_{n_i}(A,B)$ agree.

Mimicing ideas behind extended Euclide's algorithm, one can define the 
``extended subresultant pseudo-remainder sequence'' as well and obtains
recursive formulae for cofactors at the same time.
\end{rem}

Important simplifications occur in the ``normal'' case, which is the 
case where all principal subresultants do not vanish. Under this 
additional assumption, one can prove that the degrees of the $S_i$'s 
decrease by one at each step; in other words, $\deg S_i = d_B - i$ for 
all $i$. The sequence $(S_i)$ then stops at $i = d_B$. Moreover, the 
$\varepsilon_i$'s and the $c_i$'s are now all ``trivial'': we have 
$\varepsilon_i = 1$ and $c_i = s_i$ for all $i$. The recurrence formula 
then becomes:
$$S_{i+1} = s_i^2 \cdot s_{i-1}^{-2}
\cdot (S_{i-1} \,\%\, S_i) \quad \text{for } i \geq 1.$$
and Proposition \ref{prop:prem} now simply states that $R_j = 
S_{d_B-j}$. In other words, still assuming that all principal 
subresultants do not vanish, the sequence of subresultants obeys to the 
recurrence:
\begin{align}
\label{eq:recRj}
R_{d+1} = A, \quad R_d = B, & \qquad
R_{j-1} = r_j^2 \cdot r_{j+1}^{-2}
\cdot (R_{j+1} \,\%\, R_j) 
\end{align}
where $r_j$ is the leading coefficient of $R_j$ for $j \leq d$ and $r_{d+1}
= 1$ by convention. Moreover, a similar 
recurrence exists for cofactors as well:
\begin{align}
U_{d+1} = 1, \quad U_d = 0, & \qquad
U_{j-1} = r_j^2 \cdot r_{j+1}^{-2}
\cdot (U_{j+1} - Q_j U_j) \label{eq:recUj} \\
V_{d+1} = 0, \quad U_d = 1, &\qquad
V_{j-1} = r_j^2 \cdot r_{j+1}^{-2}
\cdot (V_{j+1} - Q_j V_j) \label{eq:recVj}
\end{align}
where $Q_j$ is quotient in the Euclidean division of $R_{j+1}$ by $R_j$.

Proposition \ref{prop:prem} of course yields an algorithm for computing 
subresultants. In the normal case and assuming further for simplicity 
that the input polynomials are monic of same degree, it is 
Algorithm~\ref{algo:subres}, which uses the primitive \texttt{prem} for 
computing pseudo-remainders. We recall that the pseudo-remainder of the 
division of $A$ by $B$ is the polynomial $\texttt{prem}(A,B)$ defined by 
$\texttt{prem}(A,B) = \lc(B)^{\deg B-\deg A+1} (A \% B)$ where $\lc(B)$ 
denotes the leading coefficient of $B$.

\begin{algorithm}
  \SetKwInOut{Input}{Input}
  \SetKwInOut{Output}{Output}
  \Input{Two polynomials $A, B \in K_d[X]$ (given at finie precision)}
  \Output{The complete sequence of subresultants of $A$ and $B$.}

  \BlankLine

  $R_d \leftarrow B$; $r_d \leftarrow 1$

  $R_{d-1} \leftarrow B-A$

  \For{$j = (d-1), (d-2), \ldots, 1$}
    {$r_j \leftarrow$ coefficient in $X^j$ of $R_j$

     \lIf{$r_j = 0$}{\textbf{raise} NotImplementedError}

     $R_{j-1} \leftarrow \texttt{prem}(R_{j+1}, R_j) / r_{j+1}^2$
    }

  \Return{$R_{d-1}, \ldots, R_0$}
\caption{Subresultant pseudo remainder sequence algorithm}
\label{algo:subres}
\end{algorithm}

Unfortunately, while working over a complete discrete valuation field 
$K$, the stability of Algorithm~\ref{algo:subres} is as bad as that of 
standard Euclide algorithm. The use of Algorithm~\ref{algo:subres} is 
interesting because it avoids denominators (\emph{i.e.} we always work 
over $\A$ instead $K$) but it does not improve the stability.

\begin{ex}
\label{ex:subres}
Applying Algorithm~\ref{algo:subres} with the input $(A,B)$ of 
Example~\ref{ex:Euclide}, we obtain:
\begin{align*}
R_4 & = 
  \big(29 + O(2^5)\big) X^4 + 
  \big(14 + O(2^5)\big) X^3 + 
  \big(5 + O(2^5)\big) X^2 + 
  \big(17 + O(2^5)\big) X + 
  \big(17 + O(2^5)\big) \\
R_3 & = 
  \big(4 + O(2^5)\big) X^3 + 
  \big(13 + O(2^5)\big) X^2 + 
  \big(4 + O(2^5)\big) X + 
  \big(16 + O(2^5)\big) \\
R_2 & = 
  \big(5 + O(2^5)\big) X^2 + 
  \big(20 + O(2^5)\big) X + 
  O(2^5) \\
R_1 & = 
  \big(1 + O(2)\big) X + 
  \big(1 + O(2)\big) \\
R_0 & = 1 + O(2)
\end{align*}
We observe in particular that the absolute precision on $R_0$ is $1$, 
although it should be at least $5$ since $R_0$ is given by an integral
polynomial expression in terms of the coefficients of $A$ and $B$. We
note moreover that the relative precision on $R_0$ (which is $1$ as
well) is worse that the relative precision we got on $S_7$ (which was
$3$) while executing Algorithm~\ref{algo:Euclide} (\emph{cf} 
Example~\ref{ex:Euclide}).
\end{ex}

\section{Unstability of Euclide-like algorithms}
\label{sec:unstable}

In this section, we provide strong evidences for explaining the average 
loss of precision observed while executing Algorithm~\ref{algo:subres}.
Concretely, in \S \ref{subsec:lowerbound} we establish\footnote{in a 
model of precision which is slightly weaker that the usual one; we refer 
to \S \ref{subsec:lowerbound} for a complete discussion about this.} a 
lower bound on the losses of precision which depends on extra 
parameters, that are the valuations of the principal subresultants. 
The next subsections (\S\S \ref{subsec:proba} and \ref{subsec:proof})
aim at studying the behaviour of these valuations on random inputs;
they thus have a strong probabilistic flavour.

\begin{rem}
The locution \emph{Euclide-like algorithms} (which appears in the 
title of the Section) refers to the family of
algorithms computed \textsc{gcd}s or subresultants by means of successive
Euclidean divisions. We believe that the stability of all algorithms in
this family is comparable since we are precisely loosing precision
while performing Euclidean divisions.
Among all algorithms in this family, we chose to concentrale ourselves
on Algorithm~\ref{algo:subres} because it is simpler due to the fact
that it only manipulates polynomials with coefficients in $\A$.
Nevertheless, our method extends to many other Euclide-like algorithms
including Algorithm~\ref{algo:Euclide}; this extension is left as an exercice to
the reader.
\end{rem}

\subsection{A lower bound on losses of precision}
\label{subsec:lowerbound}

We consider two fixed polynomials $A$ and $B$ with coefficients in $\A$ 
whose coefficients are known with precision $O(\pi^N)$ for some positive
integer $N$. For simplicity, we assume further that $A$ and $B$ are both
monic and share the same degree $d$. 
For any integer $j$ between $0$ and $d-1$, we denote by $R_j$ the $j$-th 
subresultant of $A$ and $B$.

In this subsection, we estimate the loss of precision if we compute the 
$R_j$'s using the recurrence \eqref{eq:recRj}.
In what follows, we are going to use a \emph{flat precision model}: this 
means that a polynomial $P(X)$ is internally represented as:
$$P(X) = \sum_{i=1}^n a_i X^i + O(\pi^N)
\quad \text{with } a_i \in K \text{ and } N \in \Z.$$
In other words, we assume that the software we are using does not
carry a precision data on each coefficient but only a unique precision
data for the whole polynomial. Concretely this means that, after having
computing a polynomial, the software truncates the precision on each
coefficient to the smallest one.
One can argue that this assumption is too strong (compared to usual
implementations of $p$-adic numbers). Nevertheless, it defines a 
simplified framework where computations can be performed and experiments 
show that it rather well reflects the behaviour of the loss of precision 
in Euclide-like algorithms.

Let $V_j$ be the valuation of the principal $j$-th subresultant of $A,B$ 
and $W_j$ be the minimum of the valuations of the coefficients of $R_j$. 
We of course have $V_j \geq W_j$ and we set $\delta_j = V_j - W_j$.

\begin{prop}
\label{prop:precEuclide}
Let $A$ and $B$ as above. Either Algorithm~\ref{algo:subres} fails or
it outputs the subresultants $R_j$'s at precision $O(\pi^{N_j})$ with:
$$N_j \leq N + V_{j+1} - 2 \cdot (\delta_{j+1} + \delta_{j+2} + \cdots
+ \delta_{d-1}).$$
\end{prop}

\begin{proof}
Using that $R_{j+1}$ and $R_j$ have the expected degrees, the
remainder $(R_{j+1} \,\%\, R_j)$ is computed as follows:
$$\begin{array}{rrcl}
\text{we set:} & S & = & R_{j+1} - r_{j+1} \cdot r_j^{-1} \cdot R_j \smallskip \\
\text{and we have:} & R_{j+1} \,\%\, R_j & = & S - s \cdot r_j^{-1} \cdot R_j
\end{array}$$
where $s$ is the coefficient of degree $j$ of $S$. Let us first estimate 
the precision of $S$. Using \eqref{eq:precmul}--\eqref{eq:precdiv}, we 
find that the computed relation precision on $r_{j+1} \cdot r_j^{-1} 
\cdot R_j$ is $\min(N_{j+1} - V_{j+1}, N_j - V_j)$. The absolute
precision of this value is then
$M = \min(N_{j+1} - \delta_j, N_j - \delta_j + V_{j+1} - V_j)$. This
quantity is also the precision of $S$ since the other summand
$R_{j+1}$ is known with higher precision. Repeating the argument, we
find that the precision of $(R_{j+1} \,\%\, R_j)$ is equal to 
$\min(M - \delta_j, N_j - \delta_j + \val(s) - V_j)$ and therefore is 
lower bounded by $M - \delta_j \leq N_j - 2 \delta_j + V_{j+1} - V_j$.
From this, we derive $N_{j-1} \leq N_j - 2 \delta_j - V_{j+1} + V_j$ and
the proposition finally follows by summing up these inequalities.
\end{proof}

The difference 
$N - N_0 = - V_1 + 2 \sum_{k=1}^d \delta_j$
is a lower bound on the number of digits lost after having computed the 
resultant using the subresultant pseudo-remainder sequence algorithm. In 
the next subsection (\emph{cf} Corollary \ref{cor:Vj}), we shall see that 
$V_1$ and all $\delta_j$'s are 
approximatively equal to $\frac 1 {p-1}$ on average. The loss of 
precision then grows linearly with respect to $d$. This confirms the 
precision benchmarks shown in Figure \ref{fig:precision}.
We emphasize one more time that this loss of precision is \emph{not}
intrinsic but an artefact of the algorithm we have used; indeed, one
should not loose any precision when computing resultants because they
are given by polynomial expressions.

\subsection{Behaviour on random inputs}
\label{subsec:proba}

Proposition \ref{prop:precEuclide} gives an estimation of the loss of 
precision in Euclide-like algorithms in terms of the quantities $V_j$ 
and $\delta_j$. It is nevertheless \emph{a priori} not clear how large 
these numbers are. The aim of this paragraph is to compute their order 
of magnitude when $A$ and $B$ are picked randomly among the set of monic 
polynomials of degree $d$ with coefficients in $\A$. In what follows, we 
assume that the residue field $k = \A/\pi\A$ is finite and we use the 
letter $q$ to denote its cardinality.

We endow $\A$ with its Haar measure. The set $\Omega$ of couples of 
monic polynomial of degree $d$ with coefficients in $\A$ is canonically 
in bijection with $\A^{2d}$ and hence inherits the product measure. 
We consider $V_j$, $W_j$ and $\delta_j$ as random variables defined on 
$\Omega$.

\begin{theo}
\label{th:lawVj}
We fix $j \in \{0, \ldots, d-1\}$.
Let $X_0, \ldots, X_{d-1}$ be $d$ pairwise independant discrete random 
variables with geometric law of parameter $(1 - q^{-1})$, \emph{i.e.}
$$\P[X_i = k] = (1 - q^{-1}) \cdot q^{-k} \quad 
\text{(with } 0 \leq i < d \text{ and } k \in \N \text{)}.$$
Then $V_j$ is distributed as the random variable
\begin{center}
$\displaystyle Y_j = \sum_{i=0}^d \min(X_{j-i}, X_{j-i+1}, \ldots, X_{j+i})$
\end{center}
\noindent
with $X_i = +\infty$ if $i < 0$ and $X_i = 0$ if $i \geq d$.
\end{theo}

\begin{rem}
The above Theorem does not say anything about the correlations between
the $X_j$'s. In particular, we emphasize that it is \emph{false} that
the tuple $(V_{d-1}, \ldots, V_0)$ is distributed as $(Y_{d-1}, \ldots,
Y_0)$. For instance, one can prove that $(V_{d-1}, 
V_{d-2})$ is distributed as 
$(X, \, X' + \min(X', [X/2]))$
where $X$ and $X'$ are two independant discrete random variables with 
geometric law of parameter $(1 - q^{-1})$ and the notation $[\cdot]$ 
stands for the integer part function. In particular, we observe that 
$(V_{d-2}, V_{d-1}) \neq (2,1)$ almost surely although the events 
$\{V_{d-1} = 2\}$ and $\{V_{d-2} = 1\}$ both occur with positive 
probability.

Nonetheless, a consequence of Proposition \ref{prop:distribk} below is 
that the variables $\bar V_j = \mathbbm{1}_{\{V_j = 0\}}$ are mutually 
independant.
\end{rem}

\begin{theo}
\label{th:deltaj}
For all $j \in \{0, \ldots, d-1\}$ and all $m \in \N$, we have:
$$\P[\delta_j \geq m] \geq \frac{(q-1)(q^j-1)}{q^{j+1}-1} q^{-m}.$$
\end{theo}

The proof of these two theorems will be given in \S \ref{subsec:proof}. 
We now derive some consequences. Let $\sigma$ denote the following
permutation:
$$\begin{array}{cl}
\left(
\begin{array}{cccccccc}
1 & 2 & \cdots & \frac d 2 & \frac d 2 + 1 & \frac d 2 + 2 & \cdots & d \smallskip \\
1 & 3 & \cdots & d-1 & d & d-2 & \cdots & 2 
\end{array} \right) & \text{if } 2 \mid d \medskip \\
\left(
\begin{array}{cccccccc}
1 & 2 & \cdots & \frac{d+1} 2 & \frac{d+3} 2 & \frac{d+5} 2 & \cdots & d \smallskip \\
1 & 3 & \cdots & d & d-1 & d-3 & \cdots & 2 
\end{array} \right) & \text{if } 2 \nmid d.
\end{array}$$
In other words, $\sigma$ takes first the odd values in $[1,d]$ in 
increasing order and then the even values in the same range in 
decreasing order.

\begin{cor}
\label{cor:Vj}
For all $j \in \{0, \ldots, d-1\}$, we have:

\begin{enumerate}[\hspace{0.3cm}(1)]
\setlength\itemsep{0.1em}
\item $\E[V_j] = 
\displaystyle \sum_{i=1}^{d-j} \frac 1 {q^{\sigma(i)} - 1}$;
in particular $\frac 1 {q-1} \leq \E[V_j] < \frac q {(q-1)^2}$
\item $\frac {q^j - 1}{q^{j+1} - 1} \leq \E[\delta_j] \leq \E[V_j]$
\item $\sigma[V_j]^2 = 
\displaystyle \sum_{i=1}^{d-j} \frac {(2i-1) \cdot q^{\sigma(i)}}
{(q^{\sigma(i)} - 1)^2}$;
in particular $\frac{\sqrt q}{q-1} \leq \sigma[V_j] < 
\frac {q \:\sqrt{q+1}} {(q-1)^2}$
\item $\P[V_j \geq m] \leq q^{-m + O(\sqrt m)}$
\item $\E[\max(V_0, \ldots, V_{d-1})] \leq \log_q d + O(\sqrt{\log_q d})$
\end{enumerate}
\end{cor}

\begin{proof}
By Theorem \ref{th:lawVj}, we have $\E[V_j] =
\sum_{i=0}^d \E[Z_i]$ with $Z_i = \min(X_{j-i}, \ldots, X_{j+i})$ ($j$ 
is fixed during all the proof). Our conventions imply that $Z_i$ 
vanishes if $i \geq d-j$. On the contrary, if $i < d-j$, let us define 
$\tau(1), \ldots, \tau(d-j)$ as the numbers $\sigma(1), \ldots, 
\sigma(d-j)$ sorted in increasing order. The random variable $Z_i$ 
is then the minimum of $\tau(i)$ independant random variables with 
geometric distribution of parameter $(1 - q^{-1})$ and thus its 
distribution is geometric of parameter $(1 - q^{-\tau(i)})$.
Its expected value is then $\frac 1 {q^{\tau(i) - 1}}$ and the first
formula follows. The inequality $\frac 1 {q-1} \leq \E[V_j]$ is clear
because $\frac 1 {q-1}$ is the first summand in the expansion of
$\E[V_j]$. The upper bound is derived as follows:
$$\E[V_j] < \sum_{i=0}^\infty \frac 1{q^i - 1}
\leq \sum_{i=0}^\infty \frac 1{q^i - q^{i-1}} = \frac q {(q-1)^2}.$$

The first inequality of claim~(2) is obtained from the relation 
$$\E[\delta_j] 
= \sum_{m=1}^\infty m \cdot \P[\delta_j = m]
= \sum_{m=1}^\infty \P[\delta_j \geq m]$$ 
using the estimation of Theorem \ref{th:deltaj}. The second inequality 
is clear because $\delta_j \leq V_j$.

The variance of $V_j$ is related to the covariance of $Z_i$'s thanks to 
the formula
$$\Var(V_j) = \sum_{1 \leq i,i' \leq d-j} \Cov(Z_i, Z_{i'}).$$
Moreover, given $X$ and $X'$ two independant variables having geometric
distribution of parameter $(1 - a^{-1})$ and $(1 - b^{-1})$ respectively,
a direct computation gives:
$$\Cov(X, \min(X,X')) = \frac{ab}{(ab-1)^2}.$$
Applying this to our setting, we get:
$$\Cov(Z_i, Z_{i'}) = \frac{q^{e(i,i')}}{(q^{e(i,i')}-1)^2}$$
where $e(i,i') = \min(\tau(i), \tau(i')) = \tau(\min(i,i'))$. Summing
up these contributions, we get the equality in~(3). The inequalities
are derived from this similarly to what we have done in~(1).

We now prove~(4). Let $(Z_i)_{i \geq 0}$ be a countable family 
of independant random variable having all geometric distribution of 
parameter $(1 - q^{-1})$. We set $Z = \sum_{i=1}^\infty \min(Z_1, \ldots, 
Z_i)$. Cleary $V_j \leq Z$ and it is then enough to prove:
$$\P[Z \geq m] \leq q^{-m + O(\sqrt m)}.$$
We introduce the event $E_m$ formulated as follows:
\emph{there exists a partition $(m_1, \ldots, m_\ell)$ of $m$ such that $X_i 
\geq m_i$ for all $i \leq \ell$}.
Up to a measure-zero subset, $E_m$ contains the event $\{ Z \geq m\}$.
We obtain this way:
$$\P[Z \geq m] \leq \P[E_m] \leq
\sum \, \prod_{i=1}^\ell \P[X_1 \geq m_i]$$
where the latter sum runs over all partitions $(m_1, \ldots, m_\ell)$ 
of $m$. Replacing $\P[X_1 \geq m_i]$ by $q^{-m_i}$, we get
$\P[E_m] \leq p(m) \cdot q^{-m}$
where $p(m)$ denotes the number of partitions of $m$. By a famous 
formula \cite{andrews}, we know that $\log p(m)$ is equivalent to $\pi 
\sqrt{2m/3}$. In particular it is in $q^{O(\sqrt m)}$ and~(4) is proved.

We now derive~(5) by a standard argument. It follows from~(4) that
$$\P[\max(V_0, \ldots, V_{d-1})] \leq d \cdot q^{-m + c\sqrt m}$$
for some constant $c$. Therefore:
$$\E[\max(V_0, \ldots, V_{d-1})] \leq \sum_{m=1}^\infty \min(1,
d \cdot q^{-m + c\sqrt m}).$$
Let $m_0$ denote the smallest index such that $d \cdot q^{-m_0 + c\sqrt 
m_0}$, \emph{i.e.} $m_0 - c \sqrt{m_0} \geq \log_q d$. Solving the latest
equation, we get $m_0 = \log_q + O(\sqrt{\log_q d})$. Moreover
$\sum_{m=m_0}^\infty d\: q^{-m + c\sqrt m}$ is bounded independantly of
$d$. The result follows.
\end{proof}

\subsection{Proof of Theorems \ref{th:lawVj} and \ref{th:deltaj}}
\label{subsec:proof}

During the proof, $A$ and $B$ will always refer to monic polynomials of 
degree $d$ and $R_j$ (resp. $U_j$ and $V_j$) to their $j$-th 
subresultant (resp. their $j$-th cofactors). If $P$ is a polynomial and 
$n$ is a positive integer, we use the notation $P[n]$ to refer to the 
coefficient of $X^n$ in $P$. We set $r_j = R_j[j]$.

\medskip

\noindent
\textbf{Preliminaries on subresultants.}
We collect here various useful relations between subresultants and 
cofactors. During all these preliminaries, we work over an arbitrary
base ring $\ring$.

\begin{prop}
\label{prop:relations}
The following relations hold:
\begin{itemize}
\setlength\itemsep{0.1em}
\item $U_{j-1} V_j - U_j V_{j-1} = (-1)^j r_j^2$;
\item $U_j[d{-}j{-}1] = -V_j[d{-}j{-}1] = (-1)^j r_{j+1}$;
\item $\Res^{j,j-1}_k(R_j, R_{j-1}) = r_j^{2(j-k-1)} \: R_k$ for $k < j$;
\item $\Res^{d-j,d-j-1}_k(U_{j-1}, U_j) = r_j^{2(d-j-k-1)} \: U_{d-1-k}$
for $k < d-j$.
\end{itemize}
Moreover $r_j$ depends only on the $2(d-j)-1$ coefficients of highest 
degree of $A$ and $B$.
\end{prop}

\begin{proof}

By functoriality of subresultants, we may assume that $\ring = \Z[a_0, 
\ldots, a_{d-1}, b_0, \ldots, b_{d-1}]$ and that $A$ and $B$ are the two 
generic monic polynomials $A = X^d + \sum_{i=0} a_i X^i$ and $B = X^d + 
\sum_{i=0} b_i X^i$. Under this additional assumption, all principal
subresultant are nonzero. Therefore, the sequences $(R_j)_j$, $(U_j)_j$
and $(V_j)_j$ are given by the recurrences \eqref{eq:recRj}--\eqref{eq:recVj}.
The two first announced relations follow easily. Let now focus on the
third one. We set $\tilde R_j = R_j$ and $\tilde R_k = r_j^{2(j-k-1)} 
\: R_k$ for $k < j$. An easy decreasing induction on $k$ shows that this 
sequence obeys to the recurrence:
$$\tilde R_{k-1} = \tilde r_k^2 \cdot \tilde r_{k+1}^{-2} \cdot 
(\tilde R_{k+1} \,\%\, \tilde R_k)$$
where $\tilde r_j = 1$ and $\tilde r_k$ is the coefficient of $\tilde 
R_k$ of degree $k$ for all $k < j$. Comparing with \eqref{eq:recRj},
this implies that $\tilde R_k$ is the $k$-th subresultant of the pair 
$(R_j, R_{j-1})$ and we are done. The fourth equality and the last
statement are proved in a similar fashion.
\end{proof}

For any fixed index $j \in \{1, \ldots, d-1\}$, we consider the 
function $\psi_j$ that takes a couple $(A,B) \in \ring_d[X]^2$ to the
quadruple $(U_j, U_{j-1}, R_j, R_{j-1})$. It follows from Proposition
\ref{prop:relations} that $\psi_j$ takes its values in the subset 
$\mathcal E_j$ of
$$\big(\ring_{\leq d-j-1}[X]\big) \times \big(\ring_{\leq d-j}[X]\big) \times 
\big(\ring_{\leq j}[X]\big) \times \big(\ring_{\leq j-1}[X]\big)$$
consisting of the quadruples $(\mathcal U_j, \mathcal U_{j-1}, \mathcal
R_j, \mathcal R_{j-1})$ such that:
$$\begin{array}{ll}
&\mathcal U_{j-1}[d{-}j] = (-1)^{j-1} \: \mathcal R_j[j] \smallskip \\
\text{and} &
\Res^{d-j,d-j-1}(\mathcal U_{j-1}, \mathcal U_j) = -\mathcal R_j[j]^{2(d-j-1)}.
\end{array}$$
Let $\mathcal E_j^\times$ be the subset of $\mathcal E_j$ defined by 
requiring that $\mathcal R_j[j]$ is invertible in $\ring$. In the same 
way, we define $\Omega_j^\times$ as the subset of $\ring_d[X]^2$ 
consisting of couples $(A,B)$ whose $j$-th principal subresultants (in 
degree $(d,d)$) is invertible in $\ring$.

\begin{prop}
\label{prop:bijection}
The function $\psi_j$ induces a bijection between
$\Omega_j^\times$ and $\mathcal E_j^\times$.
\end{prop}

\begin{proof} 
We are going to define the inverse of $\psi_j$. We fix a quadruple 
$(\mathcal U_j, \mathcal U_{j-1}, \mathcal R_j, \mathcal R_{j-1})$ in 
$\mathcal E_j^\times$ and set $a = \mathcal R_j[j]$.
Let $\mathcal W_j$ and $\mathcal W_{j-1}$ denote 
the $j$-th cofactors of $(\mathcal U_{j-1}, \mathcal U_j)$ in degree 
$(d{-}j, d{-}j{-}1)$. Define $\mathcal V_j = \alpha \mathcal W_j$ and 
$\mathcal V_{j-1} = -\alpha \mathcal W_{j-1}$ where $\alpha =
a^{4j-4d+6}$. The relation:
\begin{equation}
\label{eq:missingcof}
\mathcal U_{j-1} \mathcal V_j - 
\mathcal U_j \mathcal V_{j-1} = a^2.
\end{equation}
then holds. We now define $A$ and $B$ using the formulae:
\begin{equation}
\label{eq:systemAB}
\left\{\begin{array}{l}
A = (-1)^j \cdot a^{-2} \cdot (\mathcal V_j \mathcal R_{j-1} - \mathcal V_{j-1} \mathcal R_j) \smallskip \\
B = (-1)^{j-1} \cdot a^{-2} \cdot (\mathcal U_j \mathcal R_{j-1} - \mathcal U_{j-1} \mathcal R_j) \\
\end{array}\right.
\end{equation}
and let $\varphi_j$ be the function mapping $(\mathcal U_j, \mathcal 
U_{j-1}, \mathcal R_j, \mathcal R_{j-1})$ to $(A,B)$. The composite 
$\varphi_j \circ \psi_j$ is easily checked to be the identity: indeed, 
if $\psi_j(A,B) = (\mathcal U_j, \mathcal U_{j-1}, \mathcal R_j, 
\mathcal R_{j-1})$, the relation~\eqref{eq:missingcof} implies that 
$\mathcal V_{j-1}$ and $\mathcal V_j$ are the missing cofactors and, 
consequently, $A$ and $B$ have to be given by the 
system~\eqref{eq:systemAB}.

To conclude the proof, it remains to prove that the composite in 
the other direction $\psi_j \circ \varphi_j$ is the identity as well. 
Since both $\varphi_j$ and $\psi_j$ are componant-wise given by 
polynomials, we can use functoriality and assume that $\ring$ is 
the field $\Q(c_0, c_1, \ldots, c_n)$ (with $n = 2d$) and that 
each variable $c_i$ corresponds to one coefficient of $\mathcal U_j$, 
$\mathcal U_{j-1}$, $\mathcal R_j$ and $\mathcal R_{j-1}$ with the
convention that $c_0$ (resp. $(-1)^{j-1} c_0$) is used for the leading
coefficients of $\mathcal R_j$ (resp. $\mathcal U_{j-1}$). Set:
$$\begin{array}{ll}
& (A,B) = \varphi_j(\mathcal U_j, \mathcal U_{j-1}, 
\mathcal R_j, \mathcal R_{j-1}) \smallskip \\
\text{and} & (U_j, U_{j-1}, R_j, R_{j-1}) = \psi_j(A,B)
\end{array}$$
Since $\ring$ is a field and $\mathcal R_j[j]$ does not vanish, the 
Sylvester mapping
$$\begin{array}{rcl}
\ring_{<d-j}[X] \times \ring_{<d-j}[X] & \to & 
\ring_{<2d-j}[X]/\ring_{<j}[X] \smallskip \\
(U,V) & \mapsto & AU + BV
\end{array}$$
has to be bijective. Therefore there must exist $\lambda \in \ring$ 
such that $\mathcal R_j = \lambda \cdot R_j$ and $\mathcal U_j = \lambda 
\cdot U_j$. Similarly $(\mathcal R_{j-1}, \mathcal U_{j-1}) = \mu \cdot
(R_{j-1}, U_{j-1})$ for some $\mu \in \ring$. Identifying the leadings 
coefficients, we get $\lambda = \mu$. Writing
$\Res^{d-j,d-j-1}(\mathcal U_{j-1}, \mathcal U_j) = 
\Res^{d-j,d-j-1}(U_{j-1}, U_j)$,
we get $\lambda^{2(d-j)-1} = 1$. Since the exponent is odd, this implies 
$\lambda = 1$ and we are done.
\end{proof}

\begin{cor}
\label{cor:bijection}
We assume that $\ring = \A$. Then the map $\psi_j : \Omega_j^\times
\to \mathcal E_j^\times$ preserves the Haar measure.
\end{cor}

\begin{proof}
Proposition \ref{prop:bijection} applied with the quotient rings $\ring 
= \A/\pi^n\A$ shows that $(\psi_j \text{ mod } \pi^n)$ is a bijection 
for all $n$. This proves the Corollary.
\end{proof}

\noindent 
\textbf{The distribution in the residue field.}
We assume in this paragraph that $\ring$ is a finite field of 
cardinality $q$. We equip $\Omega_\ring = \ring_d[X]^2$ with the uniform 
distribution. For $j \in \{0, \ldots, d-1\}$ and $(A,B) \in \Omega_\ring$,
we set $\bar V_j(A,B) = 1$ if $r_j(A,B)$ vanishes and $\bar V_j(A,B) = 0$
otherwise. The functions $\bar V_j$'s define random variables over 
$\Omega_\ring$.

\begin{prop}
\label{prop:distribk}
With the above notations, the $\bar V_j$'s are mutually independant
and they all follow a Bernoulli distribution of parameter $\frac 1 q$.
\end{prop}

\begin{proof} 
Given $J \subset \{0, \ldots, d-1\}$, we denote by 
$\Omega_\ring(J)$ the subset of $\Omega_\ring$ consisting of
couples $(A,B)$ for which $r_j(A,B)$ does not vanish if and only if $j 
\in J$. We want to prove that $\Omega_\ring(J)$ has cardinality 
$q^{2d-\Card J} (q-1)^{\Card J}$. To do this, we introduce several 
additional notations. First, we write $J = \{n_1, \ldots, 
n_\ell\}$ with $n_1 > n_2 > \cdots > n_\ell$ and set $n_{\ell+1} = 
0$ by convention. Given $n$ and $m$ two integers with $m < n$, we let 
$V_{(m,n)}$ denote the set of polynomials of the form
$a_m X^m + a_{m+1} X^{m+1} \cdots + a_n X^n$
with $a_i \in \ring$ and $a_n \neq 0$. Clearly, $V_{(m,n)}$ has
cardinality $(q-1) q^{n-m}$. If $P$ is any polynomial of 
degree $n$ and $m < n$ is an integer, we further define $P[m{:}] \in 
V_{(m,n)}$ as the polynomial obtained from $P$ by removing its monomials 
of degree $< m$. Finally, given $(A,B)$ in $\Omega_\ring$, we 
denote by $(S_i(A,B))$ its subresultant pseudo-remainder sequence as 
defined in \S \ref{subsec:subres}. We note that, if $(A,B) \in 
\Omega_\ring(J)$, the sequence $(S_i(A,B))$ stops at $i = \ell$ and we 
have $\deg S_i = n_i$ for all $i$. We now claim that the mapping
$$\begin{array}{rcl} 
\Lambda_J : \,\, 
\Omega_\ring(J) & \to & 
V_{(n_1,n_2)} \times \cdots \times V_{(n_\ell,n_{\ell+1})} \smallskip \\
(A,B) & \mapsto & 
\big(S_i(A,B)[n_{i+1}{:}]\big)_{1 \leq i \leq \ell}
\end{array}$$ 
is injective. In order to establish the claim, we remark that the 
knowledge of $S_{i-1}(A,B)$ and $S_i(A,B)[n_{i+1}{:}]$ (for some $i$) 
is enough to reconstruct the quotient of the Euclidean division of 
$S_i(A,B)$ by $S_{i-1}(A,B)$. Thus, one can reconstruct $S_i(A,B)$ from 
the knowledge of 
$S_{i-2}(A,B)$, $S_{i-1}(A,B)$ and $S_i(A,B)[n_{i+1}{:}]$. We deduce 
that $\Lambda_J(A,B)$ determines uniquely all $S_i(A,B)$'s 
and finally $A$ and $B$ themselves. This proves the claim.

To conclude the proof, we note that the claim implies that the cardinality of 
$\Omega_\ring(J)$ is at most $q^{2d-\ell} (q-1)^\ell$. Summing up these 
inequalities over all possible $J$, we get $\Card \Omega_\ring \leq 
q^{2d}$. This latest inequality being an equality, we must have
$\Card \Omega_\ring(J) = q^{2d-\Card J} (q-1)^{\Card J}$ for all $J$.
\end{proof}

\noindent
\textbf{Proof of Theorem \ref{th:deltaj}.}
We assume first that $j < d-1$. Proposition \ref{prop:distribk} above 
ensures that $r_{j+1}$ is invertible in $\A$ with probability 
$(1 - q^{-1})$. Moreover, assuming that this event holds, Corollary 
\ref{cor:bijection} implies that $R_j$ is distributed in $\A_{\leq 
j}[X]$ according to the Haar measure. An easy computation gives
$\P[\delta_j \geq m\,|\, r_{j+1} \in \A^\times ] =
\frac{q(q^j - 1)}{q^{j+1}-1}$
and therefore:
$$\P[\delta_j \geq m] \geq (1 - q^{-1}) \cdot
\frac{q(q^j - 1)}{q^{j+1}-1} =
\frac{(q-1)(q^j - 1)}{q^{j+1}-1}.$$
The case $j = d-1$ is actually simpler. Indeed, the same argument works 
expect that we know for sure that $r_{j+1} = r_d$ is invertible since it 
is equal to $1$ by convention. In that case, the probability is then 
equal to $\frac{q(q^j - 1)}{q^{j+1}-1}$.

\medskip

\noindent
\textbf{Proof of Theorem \ref{th:lawVj}.}
We fix $j \in \{0, \ldots, d-1\}$. 
We define the random variable $V_j^{(0)}$ as the greatest (nonnegative) 
integer $v$ such that all principal subresultants $r_{j'}$ have positive 
valuation for $j'$ varying in the open range $(j-v, j+v)$ (with the
convention that $r_{j'} = 0$ whenever $j' < 0$). It is clear from the 
definition that $r_{j-v}$ or $r_{j+v}$ (with $v = V_j^{(0)}$) has 
valuation $0$. Moreover, assuming first that $\val(r_{j+v}) = 0$, we 
get by Proposition \ref{prop:relations}:
$$\begin{array}{rl}
& \val(r_j) = v + \val\big(r_v^{j-v,j-v+1}(A^{(0)}, B^{(0)})\big) 
\smallskip \\
\text{with} &
A^{(1)} = \frac 1 {r_{j+v} X^{j-v-1}} \cdot R_{j+v}[j{-}v{-}1\:{:}], \smallskip \\
\text{and} &
B^{(1)} = A^{(1)} + \frac 1 {\pi X^{j-v-1}} \cdot R_{j+v-1}[j{-}v{-}1\:{:}]
\end{array}$$
where we recall that, given a polynomial $P$ and an integer $m$, the 
notation $P[m{:}]$ refers to the polynomial obtained from $P$ by
removing its monomials of degree strictly less than $m$.
We notice that all the coefficients of $B^{(1)}$ lie in $\A$ 
because $r_{j'}$ has positive valuation for $j' \in (j-v, j+v)$.
Furthermore, Corollary \ref{cor:bijection} shows that the 
couple $(A^{(1)}, B^{(1)})$ is distributed according to the Haar measure 
on $(\A_{2v-1}[X])^2$. If $\val(r_{j+v}) = 0$, one can argue similarly 
by replacing $R_{j+v}$ and $R_{j+v-1}$ by the cofactors $U_{j-v}$ and 
$U_{j-v+1}$ respectively. Replacing $(A,B)$ by $(A^{(1)}, B^{(1)})$, we 
can now define a new random variable $V_j^{(1)}$ and, continuing this 
way, we construct an infinite sequence $V_j^{(m)}$ such that
$V_j = \sum_{m \geq 0} V_j^{(m)}$.

We now introduce a double sequence $(X_i^{(m)})_{0 \leq i < d, m 
\geq 0}$ of mutually independant random variables with Bernoulli
distribution of parameter $\frac 1 q$ and we agree to set 
$X_{j'}^{(m)} = 0$ for $j' < 0$ and $X_{j'}^{(m)} = 1$ for $j \geq d$.
It follows from Proposition \ref{prop:distribk} (applied with $\ring = 
k$) that $V_j^{(0)}$ has the same distribution than 
$Y_j^{(0)} = \sum_{i=1}^d \min(X_{j-i}^{(0)}, \ldots, X_{j+i}^{(0)})$.
In the same way, keeping in mind that $A^{(1)}$ and $B^{(1)}$ have both 
degree $2V_j^{(0)} - 1$, we find that $V_j^{(1)}$ has the same 
distribution than $\sum_{i=1}^{V_j^{(0)} - 1} \min(X_{j-i}^{(1)}, 
\ldots, X_{j+i}^{(1)})$, which can be rewritten as
$Y_j^{(1)} = \sum_{i=1}^d \min(X_{j-i}^{(0)}, X_{j-i}^{(1)}, \ldots, 
X_{j+i}^{(0)}, X_{j+i}^{(1)})$. More precisely,
the equidistribution of $(A^{(1)}, B^{(1)})$ shows that the 
joint distribution $(V_j^{(0)}, V_j^{(1)})$ is the same as those
of $(Y_j^{(0)}, Y_j^{(1)})$. Repeating the argument, we see that
$(V_j^{(m)})_{m \geq 0}$ is distributed as $(Y_j^{(m)})_{m \geq 0}$
where:
$$Y_j^{(m)} = \sum_{i=1}^d \min(X_{j-i}^{(0)}, \ldots X_{j-i}^{(m)}, 
\ldots, X_{j+i}^{(0)}, \ldots, X_{j+i}^{(m)}).$$
Setting finally $X_i = \sum_{m \geq 0} \min(X_1^{(0)}, \ldots, 
X_i^{(m)})$, we find the $X_i$'s ($0 \leq i < d$) are mutually 
independant and that they all follow a geometric distribution of 
parameter $(1 - q^{-1})$. We now conclude the proof by noting that 
$Y_j$ equals $\sum_{i=1}^d \min(X_{j-i}, \ldots, X_{j+i})$ (recall
that the $X_i^{(m)}$'s only take the values $0$ and $1$).

\section{A stabilized algorithm for computing subresultants}
\label{sec:stable}

We have seen in the previous sections that Euclide-like algorithm are 
unstable in practice. On the other hand, one can compute subresultants 
in a very stable way by evaluating the corresponding minors of the 
Sylvester matrix. Doing so, we do not loose any significant digit. Of 
course, the downside is the rather bad efficiency.

In this section, we design an algorithm which combines the two
advantages: it has the same complexity than Euclide's algorithm
and it is very stable in the sense that it does not loose any
significant digit. 
This algorithm is deduced from the subresultant pseudo-remainder
sequence algorithm by applying a ``stabilization process'', whose 
inspiration comes from \cite{padicprec}.

\subsection{Crash course on ultrametric precision}
\label{subsec:crashcourse}

In this subsection, we briefly report on and complete the results of 
\cite{padicprec} where the authors draw the lines of a general framework 
to handle a sharp (often optimal) track of ultrametric precision. In 
what follows, the letter $W$ still refers to a complete DVR while the 
letter $K$ is used for its fraction field.

\subsubsection{The notion of lattice}

As underlined in Remark~\ref{rem:balls}, the usual way of tracking 
precision consists in replacing elements of $\A$ --- which cannot fit 
entirely in the memory of a computer --- by balls around them. Using 
this framework, a software manipulating $d$ variables in $\A$ will work 
with $d$ ``independant'' balls. The main proposal of \cite{padicprec} is 
to get rid of this ``independance'' and model precision using a unique
object contained in a $d$-dimensional vector space. In order to be more
precise, we need the following definition.

\begin{deftn}
A $\A$-lattice in a finite dimensional vector space $E$ over $K$ is
a $\A$-submodule of $E$ generated by a $K$-basis of $E$.
\end{deftn}

Although the defintion of a lattice is similar to that of $\Z$-lattice 
in $\R^d$, the geometrical representation of it is quite different. 
Indeed, the elements of $\A$ themselves are not distributed as $\Z$ is 
in $\R$ but rather from a ball inside $K$ (they are exactly elements of 
norm $\leq 1$). More 
generally, assume that $E$ is equipped with a ultrametric norm $\Vert 
\cdot \Vert_E$ compatible with that on $K$ (\emph{i.e.} $\Vert \lambda x 
\Vert_E = |\lambda| \cdot \Vert x \Vert_E$ for $\lambda \in K$, $x \in 
E$). (A typical example is $E = K^n$ equipped with the sup norm.) One
checks that the balls
$$B_E(r) = \big\{ \,\, x \in E \quad\big|\quad \Vert x \Vert_E \leq r \,\,\big\}$$
are all lattices in $E$. Moreover, any lattice is deduced from $B_E(1)$ 
by applying a bijective linear endomorphism of $E$. Therefore, lattices 
should be thought as special neighborhoods of $0$ (see Figure~\ref{fig:lattice}).
\begin{figure}
\hfill
\begin{tikzpicture}
\begin{scope}[beige, rounded corners=10pt]
\fill (0.1,0.2) rectangle (3.9,5.8);
\end{scope}

\begin{scope}[xshift=2cm,yshift=3cm,rotate=20]
\draw[black,fill=green!50] (-0.7,-1.5) rectangle (0.7,1.5);
\draw[black,->] (-1,0)--(1,0);
\draw[black,->] (0,-2)--(0,2);
\end{scope}

\begin{scope}[very thick, rounded corners=10pt]
\draw (0.1,0.2) rectangle (3.9,5.8);
\end{scope} 
\end{tikzpicture}
\hfill \null

\caption{Picture of a lattice in the ultrametric world}
\label{fig:lattice}
\end{figure}
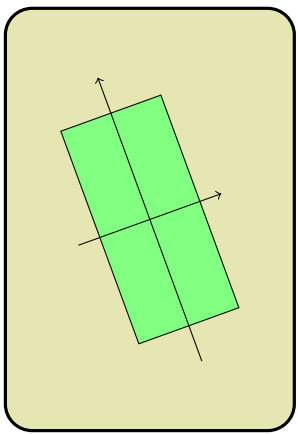
As a consequence, cosets of the form $x + H$, where $H$ is a lattice, 
appear as interesting candidates to model precision. This feeling is 
consolidated by the following result which roughly speaking claims that 
such cosets behave quite well under differentiable maps.

\begin{lem}[\cite{padicprec}, Lemma 3.4]
\label{lem:padicprec}
Let $E$ and $F$ be two normed finite dimensional $K$-vector spaces.
Let $f : E \to F$ be a function of class $C^1$ and let $x$
be a point in $K^n$ at which the differential of $f$, denoted by
$f'(x)$, is surjective.
Then, for all $\rho \in (0,1]$, there exists $\delta > 0$ such that
the following equality holds:
\begin{equation}
\label{eq:padicprec}
f(x + H) = f(x) + f'(x)(H)
\end{equation}
for any lattice $H$ satisfying $B_E(\rho r) \subset H \subset B_E(r)$
for some $r < \delta$.
\end{lem}

In what follows, we will often use Lemma \ref{lem:padicprec} with
$\rho = 1$. It states in this particular case that 
\begin{equation}
\label{eq:padicprec2}
f(x + B_E(r)) = f(x) + f'(x)(B_E(r))
\end{equation}
as soon as $r$ is small enough. It is moreover possible to provide
an explicit upper bound on $r$ assuming that $f$ has more regularity.
The case of locally analytic functions is treated in \cite{padicprec}
in full generality. Nevertheless, for the application we have in mind, 
it will be enough to restrict ourselves to the simpler case of 
\emph{integral polynomial} functions. In order to proceed, we assume 
that $E$ is endowed with distingushed ``orthonormal'' 
basis\footnote{One can prove that such a basis always exists.}, that is 
a basis $(e_1, \ldots, e_n)$ with the property that $\Vert \sum_{i=1}^n 
x_i e_i \Vert_E = \max_{1\leq i \leq n} |x_i|$ for all 
families of $\lambda_i$'s lying in $K$. In other words, the choice of
this distingushed ``orthonormal'' basis defines a norm-preserving 
isomorphism between $E$ and $K^n$ endowed with the sup norm. We assume
similarly that we are given a distingushed ``orthonormal'' basis
$(f_1, \ldots, f_m)$ of $F$. Then any function $f : E \to F$ can
be written in our distinguished system of coordinates as follows:
$$f(x) = \sum_{j=1}^m F_j(x_1, \ldots, x_n) f_j
\quad \text{with} \quad x = \sum_{i=1}^n x_i e_i.$$

\begin{deftn}
The function $f$ is \emph{integral polynomial} if all $F_j$'s are
polynomials functions with coefficients in $\A$.
\end{deftn}

\begin{ex}
\label{ex:Gaussnorm}
Let us examine more closely the case of polynomial spaces since it will 
be considered repeadtly in the sequel. We take $E = K_{<n}[X]$ and $F = 
K_{<m}[X]$ and endow both with the Gauss norm, which is defined by:
\begin{align*}
\Vert a_0 + a_1 X + \cdots + a_{n-1} X^{n-1} \Vert_E & 
  = \max \big(|a_0|, |a_1|, \ldots, |a_{n-1}|\big) \\
\Vert b_0 + b_1 X + \cdots + b_{m-1} X^{m-1} \Vert_F & 
  = \max \big(|b_0|, |b_1|, \ldots, |b_{m-1}|\big)
\end{align*}
It is clear from these definitions that the canonical basis $(1, X, 
\ldots, X^{n-1})$ and $(1, X, \ldots, X^{m-1})$ of $E$ and $F$ 
respectively are ``orthonormal''. Moreover the coordinates in these
basis are the $a_i$'s and the $b_i$'s respectively. Hence, an integral
polynomial function $f : E \to F$ is nothing but a function mapping a
a polynomial $P$ to a polynomial $Q$ whose coefficients are given by
polynomial expressions which involve only the coefficients of $P$ and 
some constants in $\A$.
\end{ex}

Obviously, all integral polynomial functions are function of class $C^1$ 
(and even locally analytic), so that Lemma~\ref{lem:padicprec} applies 
to them. Proposition \ref{prop:padicprec} below exhibits an explcit 
value for the bound $\delta$ appearing in Lemma~\ref{lem:padicprec} when 
$f$ is integral polynomial and $r=1$.

\begin{prop}
\label{prop:padicprec}
Let $f : E \to F$ be an integral polynomial function and $x \in B_E(1)$. 
Then, Eq.~\eqref{eq:padicprec2} holds as soon as 
$B_F(r) \subset f'(x)(B_E(1))$.
\end{prop}

\begin{proof}
It is a direct corollary of \cite[Proposition 3.12]{padicprec}.
\end{proof}

\subsubsection{Application to precision}

Let us now briefly explain how Lemma~\ref{lem:padicprec} can be utilized 
for tracking precision.

\paragraph{Tracking precision \emph{locally}}

Assume first that we want to perform a given rather simple operation --- 
corresponding, say, to an elementary step (\emph{e.g.} an iteration of 
the main loop) of the algorithm we are executing --- modeled by a 
function $g$ of class $C^1$ defined on an open subset $U$ of a finite 
dimensional normed $K$-vector space $E$ and taking values in another 
finite dimensional normed $K$-vector space $F$.
Our input is an approximated element of $U$ which is represented 
by a coset $C$ with respect to some lattice $H$, that is a subset of $U$ 
of the form $C = x+H$ for some $x \in U$. We would like to insist on the 
following: the value of $x$ is a priori \emph{not} given; only is given 
the subset $C$. However, since $H$ is stable under addtion, we have $C = 
x+H$ for \emph{any} element $x \in C$.\footnote{This assertion means that 
any element of the ``rectangle'' $C$ is a center of it... which might be 
surprising if we are accustomed to real numbers.} As explained in \S 
\ref{subsubsec:computDVR}, assuming that $g$ is given as an algebraic 
expression, the naive solution for evaluating $g(C)$ consists in using 
formulas \eqref{eq:precadd}--\eqref{eq:precdiv}. However, this often 
results in an overestimation on the precision, in the following sense: 
this method leads to some inclusion
$$g(C) = g(x+H) \subset y + H_{\text{naive}}$$
where $y \in F$ and $H_{\text{naive}}$ is a lattice which is generally 
much more larger that $g'(x)(H)$, the latter being the best possible
one according to
Lemma~\ref{lem:padicprec} (assuming that the assumptions of this Lemma
are fullfiled). In order to avoid this and be sharp on
precision, another solution consists in splitting the computation of 
$g(C)$ into two parts as follows:
\begin{enumerate}[(A)]
\item \label{item:gpx} compute $g'(x)(H)$, and
\item \label{item:gx} compute $g(x)$ for some $x \in C$.
\end{enumerate}
Part~\eqref{item:gpx} is not easy to handle in full generality: in order 
to be efficient, a special close analysis taking advantage of the 
particular problem under consideration is often necessary. For now, let 
us simply assume that we have given two lattices $H_\min$ and $H_\max$ 
with the property that:
\begin{equation}
\label{eq:Hminmax}
H_\min \subset g'(x)(H) \subset H_\max.
\end{equation}
We shall see later (\emph{cf} \S \ref{subsec:stabilization}) how these
lattices can be constructed --- for a negligible cost --- in the special 
case of subresultants. 

We now focus on part~\eqref{item:gx}, which also requires some 
discussion. Indeed, computing $g(x)$ is not straightforward because $x$ 
itself lies in a $K$-vector space and therefore cannot be stored and 
manipulated on a computer. Nevertheless, one can take advantage of the 
fact that $x$ may be chosen arbitrarily in $C$. More precisely, we pick 
a sublattice $H'$ of $H$ and consider the new approximated element $x+H' 
\subset x+H$. Concretely, this means that we arbitrarily increase the 
precision on the given input $x$. Now, applying the naive method with 
$x+H'$, we compute some $y \in F$ and some lattice $H'_{\text{naive}} 
\subset F$ with the property that:
$$g(x+H') \subset y + H'_{\text{naive}}.$$
If furthemore $H'$ is chosen in such a way that $H'_{\text{naive}} 
\subset H_\min$, the two cosets $y + g'(x)(H)$ and $g(C)$ have a 
non-empty intersection because $g(x)$ lies in both. Therefore they 
must coincide. We deduce that $y \in g(C)$. This exactly means that
$y$ is an acceptable value for $g(x)$ and we are done.
Moreover, estimating the dependance of $H'_{\text{naive}}$ in terms 
of $H'$ is usually rather easy (remember that $g$ is supposed to model
a simple operation). Hence since $H_\min$ is known --- as
we had assumed --- finding $H'$ satisfying the required assumption is
generally not difficult. 

\begin{figure}
\hfill
\begin{tikzpicture}[scale=1.2]

\begin{scope}[beige, rounded corners=10pt]
\fill (0.1,0.2) rectangle (3.9,5.8);
\fill (6.1,0.2) rectangle (9.9,5.8);
\end{scope}

\draw[->,very thick] (4.1,3)--(5.9,3)
  node[midway,above] { $g$ };

\begin{scope}[xshift=2cm,yshift=3.5cm,rotate=20]
\draw[black,fill=green!50] (-0.8,-1.8) rectangle (0.8,1.8);
\node[scale=0.75] at (0,0) { $x{+}H$ };
\end{scope}
\draw[black,fill=purple!50] (2,2.8) rectangle (2.8,2);
\node[scale=0.75] at (2.4,2.4) { $x{+}H'$ };

\begin{scope}[xshift=8cm,yshift=3cm]
\draw[black,fill=orange!50] (-1.7,-2.3) rectangle (1.7,1.8);
\node[above left,scale=0.75] at (1.7,-2.3) { $g(x){+}H_\max$ };
\draw[black,fill=green!50,rotate=-40] (-1,-1.3) rectangle (1,1.3);
\node[scale=0.75] at (0.15,1.2) { $g(x){+}{}$ };
\node[scale=0.75] at (0.1,0.9) { $g'(x)(H)$ };
\draw[black,fill=red!50] (-0.8,-0.8) rectangle (0.4,0.4);
\node[scale=0.75] at (-0.2,0.1) { $g(x){+}H_\min$ };
\draw[black,fill=purple!50] (-0.7,-0.7) rectangle (0.3,-0.3);
\node[scale=0.75] at (-0.2,-0.5) { $y{+}H'_{\text{naive}}$ };
\end{scope}

\begin{scope}[very thick, rounded corners=10pt]
\draw (0.1,0.2) rectangle (3.9,5.8);
\draw (6.1,0.2) rectangle (9.9,5.8);
\end{scope} 
\end{tikzpicture}
\hfill \null

\caption{Method for tracking precision based on Lemma~\ref{lem:padicprec}}
\label{fig:adaptative}
\end{figure}
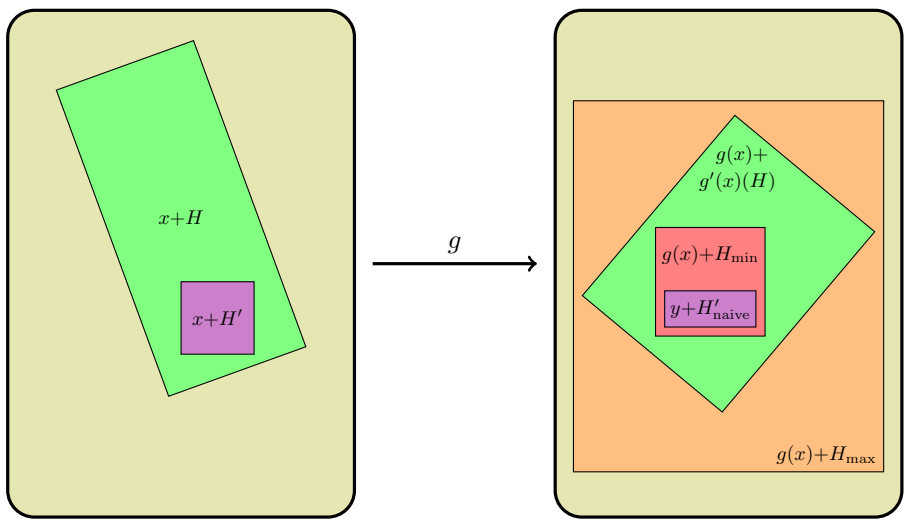

\paragraph{Tracking precision \emph{globally}}

As already said, we shall use the above method for tracking precision 
while executing a single step in a complete algorithm. Let us now 
address the problem of ``glueing''. We consider an algorithm \texttt{F} 
consisting in a succession of $n$ steps $\texttt{G}_0, \ldots, 
\texttt{G}_{n-1}$. It is modeled by a function $f : U \to F$ of class 
$C^1$ where $U$ is an open subset in a finite dimensional normed 
$K$-vector space $E$ and $F$ is a finite dimensional normed $K$-vector 
space. The input of \texttt{F} is an approximated element in $U$ 
represented as a coset $C = x+H$ where $x \in U$ and $H$ is a lattice.
We also introduce notations for each individual step. For all $i$, we
assume that $\texttt{G}_i$ is modeled by a function $g_i : U_i \to
U_{i+1}$ of class $C^1$ where $U_i$ is an open subset is 
some normed $K$-vector space $E_i$ and, by convention, $U_0 = U$, $E_0 
= E$ and $U_n = E_n = F$. We thus have:
  
$$f = g_{n-1} \circ g_{n-2} \circ \cdots \circ g_1 \circ g_0.$$
For all $i$, we set $f_i = g_{i-1} \circ \cdots \circ g_0$. It is the 
function modeling the execution of the $i$ first steps of our algorithm. 
We further define $x_i = f_i(x)$ and $H_i = f_i'(x)(H)$. The chain rule 
for composing differentials readily implies the recurrence
\begin{equation}
\label{eq:giprime}
H_{i+1} = g_i'(x_i)(H_i)
\end{equation}
For simplicity, we make the following assumptions:
\begin{itemize}
\item the $\Zp$-submodule $H_i$ is a lattice in $E_i$ such that $x_i + 
H_i \subset U_i$;
\item the triple $(g_i, x_i, H_i)$ satisfies the assumptions of 
Lemma~\ref{lem:padicprec};
\item for all $i$, we have succeeded in finding (good enough) explicit 
lattices $H_{\min,i}$ and $H_{\max,i}$ such that $H_{\min,i} \subset H_i 
\subset H_{\max,i}$;
\item for all $i$, we have succeeded in finding an explicit lattice
$H'_i$ such that, while tracking naively precision, we end up with an
inclusion
$$g_i(x_i + H'_i) = x_{i+1} + H_{\text{naive},i+1}$$
with $H_{\text{naive},i+1} \subset H_{\min,i+1}$.
\end{itemize}
We note that the first and the second assumptions are quite strong 
because they imply in particular that the sequence of $\dim E_i$ is 
non-increasing. However, it really simplifies the forthcoming discussion 
and will be harmless for the application developed in this paper. As
already mentionned, the construction of $H_{\min,i}$ and $H_{\max,i}$
will generally follow from a \emph{theoretical} argument depending on
the setting, while exhibiting $H'_i$ will often be straightforward.
Anyway, we are now in position to apply the method for 
tracking precision locally we have discussed earlier to all $g_i$'s. 
This leads to a \emph{stabilized} version of the algorithm \texttt{F} 
whose skeleton is depicted in Algorithm~\ref{algo:stabilizedF}.

\begin{algorithm}
  \SetKwInOut{Input}{Input}
  \SetKwInOut{Output}{Output}
  \Input{$x$ given at precision $O(H)$}
  \Output{$g(x)$ given at precision $O(H_{\max,n})$}
  
  \BlankLine

  $x_0 \leftarrow x$

  \For{$i = 0, \ldots, n-1$}
    {\textbf{lift} $x_i$ to precision $O(H'_i)$

     $x_{i+1} \leftarrow \texttt{G}_i(x_i)$}

  \Return{$x_n + O(H_{\max,n})$}
\caption{Stabilized version of \texttt{F}}
\label{algo:stabilizedF}
\end{algorithm}

The correctness of Algorithm~\ref{algo:stabilizedF} (under the 
assumptions listed above) is clear after Lemma~\ref{eq:padicprec}.

\subsection{Application to subresultants}
\label{subsec:stabilization}

We now apply the theory presented in \S \ref{subsec:crashcourse} above 
to the problem of computing subresultants, \emph{i.e.} the abstract 
Algorithm \texttt{F} is now instantiated to Algorithm~\ref{algo:subres}. 
We split this algorithm into steps in the obvious manner, each step 
corresponding to an iteration of the main loop. We thus consider the
functions:

$$\begin{array}{rrclc}
&g_d \, : \, K_d[X] \times K_d[X] &\to &K_d[X] \times K_{\leq d-1}[X] 
\smallskip \\
&(A,B) & \mapsto & (B, A-B) \bigskip \\
\text{and} &g_j \, : \, 
K_{\leq j+1}[X] \times K_{\leq j}[X]
& \to & K_{\leq j}[X] \times K_{\leq j-1}[X] \smallskip \\
&(R_{j+1}, R_j) & \mapsto & (R_j, R_{j-1})
& 
\end{array}$$
where $R_{j-1}$ is defined as usual by
$R_{j-1} = r_j^2 \cdot r_{j+1}^{-2} \cdot 
(R_{j+1} \,\%\, R_j)$
where $r_j$ (resp. $r_{j+1}$) stands for the coefficient of degree $j$ 
in $R_j$ (resp. of degree $j+1$ in $R_{j+1}$). We remark that $g_j$ is 
only defined on the subset consisting of pairs $(R_{j+1}, R_j)$ for 
which $R_{j+1}$ has degree $j+1$; this reflects the fact that 
Algorithm~\ref{algo:subres} fails on inputs for which at least one 
principal
subresultant vanishes. The composite function $f = g_1 \circ \cdots 
\circ g_d$ (be careful with the order of the indices) models (a slight 
variant of) Algorithm~\ref{algo:subres}. For all $j$, we put $f_j = 
g_{j+1} \circ \cdots \circ g_d$; it is the function:
$$\begin{array}{rcl}
f_j \, : \, K_d[X] \times K_d[X] & \to & K_{\leq j}[X] \times 
K_{\leq j-1}[X] \smallskip \\
(A,B) & \mapsto & (\Res_j(A,B), \Res_{j-1}(A,B)).
\end{array}$$

For simplicity, we assume in addition that the precision on the input 
$(A,B)$ is \emph{flat}, meaning that all coefficients of $A$ and $B$ are 
known with the same absolute precision $N$. In the language of \S 
\ref{subsec:crashcourse}, this flat precision corresponds to the lattice 
$H = \pi^N \calL$ where $\calL = \A_{<d}[X] \times \A_{<d}[X]$ is the 
unit ball in $K_d[X] \times K_d[X]$ with respect to the Gauss norm
(\emph{cf} Example~\ref{ex:Gaussnorm}). Following \S 
\ref{subsec:crashcourse}, our first task consists in finding two 
lattices $H_{\min,j}$ and $H_{\max,j}$ having the property that 
$H_{\min,j} \subset f'_j(A,B)(H) \subset H_{\max,j}$.

\begin{lem}
\label{lem:diff}
For all $(A,B) \in K_d[X]^2$, we have:
$$r_j^2 \cdot \mathcal L_j 
\subset f'_j(A,B)(\mathcal L) \subset \mathcal L_j$$
where $r_j$ is the $j$-th principal subresultant of $(A,B)$ and
$\calL_j = \A_{\leq j}[X] \times \A_{\leq j-1}[X]$ is the unit ball
in $K_{\leq j}[X] \times K_{\leq j-1}[X]$.
\end{lem}

\begin{proof}
The second inclusion is clear because $f_j$ is a polynomial function. 
Let us prove the first inclusion. One may of course assume that $r_j$ 
does not vanish, otherwise there is nothing to prove. Now, we remark 
that $f_j$ factors through the function $\psi_j$ introduced in \S 
\ref{subsec:proof}. By continuity, the $j$-th principal subresultant 
function does not vanish on a neighborhood of $(A,B)$. By Proposition 
\ref{prop:bijection}, $\psi_j$ is injective on this neighborhood. 
Therefore so is $f_j$. Furthermore, a close look at the proof of 
Proposition \ref{prop:bijection} indicates that a left inverse of $f_j$ 
is the function mapping $(S_j, S_{j-1})$ to
$$(-1)^j \cdot r_j^{-2} \cdot 
(V_j S_{j-1}{-}V_{j-1} S_j,\: -U_j S_{j-1}{+}U_{j-1} S_j)$$
where $U_j, V_j$ (resp. $U_{j-1}$, $V_{j-1}$) are the $j$-th (resp 
$(j-1)$-th) cofactors of $(A,B)$. Differenting this, we get the
announced result.
\end{proof}

Lemma \ref{lem:diff} ensures that one can safely take $H_{\min,j} = 
r_j^2 \cdot \pi^N \calL_j$ and $H_{\max,j} = \pi^N \calL_j$. It finally 
remains to construct the lattice $H'_j \subset K_{\leq j}[X] \times 
K_{\leq j-1}[X]$. For this, we remark that a naive track of precision 
leads to a loss of at most $2 \cdot \val(r_{j+1})$ digits while 
executing the step $\texttt{G}_j$ (see also proof of Proposition 
\ref{prop:precEuclide} for similar considerations). Therefore, one can 
take $H'_j = r_j^2 r_{j+1}^2 \cdot \pi^N \calL_j$. Instantiating
Algorithm~\ref{algo:stabilizedF} in this particular case, we end up
with Algorithm~\ref{algo:stabsubres} below which then appears as a 
stable version of Algorithm~\ref{algo:subres}.

\begin{algorithm}
  \SetKwInOut{Input}{Input}
  \SetKwInOut{Output}{Output}
  \Input{Two polynomials $A, B \in K_d[X]$ given at flat precision $O(\pi^n)$}
  \Output{The sequence of subresultants of $A$ and $B$ given at flat precision $O(\pi^n)$}

  \BlankLine

  $R_d \leftarrow B$; $r_d \leftarrow 1$

  $R_{d-1} \leftarrow B-A$

  \For{$j = (d-1), (d-2), \ldots, 1$}
    {$r_j \leftarrow$ coefficient in $X^j$ of $R_j$

     \lIf{$v_j \geq \frac N 2$}{\textbf{raise} NotImplementedError}

     \textbf{lift} $(R_{j+1}, R_j)$ at precision $O(\pi^{N+2\val(r_j) + 2\val(r_{j+1}))})$

     $R_{j-1} \leftarrow \texttt{prem}(R_{j+1}, R_j) / r_{j+1}^2$
    }

  \Return{$R_{d-1} + O(\pi^N), \ldots, R_0 + O(\pi^N)$}
\caption{Stabilized version of Algorithm~\ref{algo:subres}}
\label{algo:stabsubres}
\end{algorithm}

\begin{prop}
Algorithm \ref{algo:stabsubres} computes all subresultants of
$(A,B)$ at precision $O(\pi^N)$ under the following 
assumption\footnote{If this assumption is not fullfiled, the
algorithms fails and returns an error.}

\medskip

\begin{tabular}{rl}
{\bf (H)}: & all principal subresultants of $(A,B)$ 
do not vanish modulo $\pi^{N/2}$.
\end{tabular}

\medskip

\noindent
It runs in $O(d^2 \cdot \M(N + \max(V_0, \ldots, V_{d-1}))$ bit 
operations where $V_j$ denotes the valuation of $r_j$ and $\M(n)$ is the 
number of bit operations needed to perform an arithmetic operation 
(addition, product, division) in $\A$ at precision $O(\pi^n)$.
\end{prop}

\begin{rem}
\label{rem:M}
In all usual examples ($p$-adic numbers, Laurent series), one can 
choose $\M(n)$ to be quasi-linear in $n$ and the size of the residue
field $k$.
\end{rem}

\begin{proof}
Correctness has been already proved (the assumption {\bf (H)} ensures
that Proposition~\ref{prop:padicprec} applies to each $g_j$).
As usual Euclide's algorithm, Algorithm 1 requires $O(d^2)$ operations 
in the base ring $\A$. Moreover, we observe that the maximal precision 
at which we are computing is upper bounded by $N + 2 \max(V_0, \ldots, 
V_{d-1})$. This justifies the announced complexity.
\end{proof}

According to Corollary \ref{cor:Vj}, the expected value of the variable 
$\max(V_0, \ldots, V_{d-1})$ is in $O(\log_p d)$. Thus, the average 
complexity of Algorithm 1 is $O(d^2 \cdot \M(N + \log d))$ bit 
operations. In all usual cases (\emph{cf} Remark \ref{rem:M}), this
complexity is also $\tilde O(d^2 N \cdot \log |k|)$ bit operations.

To conclude with, let us comment on briefly the hypothesis {\bf (H)}. We 
first remark that it is satisfied with high probability if $N$ is large 
compared to $2 \cdot \log_d p$. Thus, replacing eventually $N$ by $3 
\cdot \log_d p$ (which does not affect the complexity), the 
assumption~(b) is harmless on average --- but maybe not on particularly 
bad instances. We moreover underline that, if we are just interested in 
computing the $j$-th subresultant for a particular $j$, then we just 
need to assume the non-vanishing of the principal subresultants in the 
range $[j+1, d-1]$.

\subsubsection*{Open questions}
\label{subsec:questions}

The first hypothesis we would like to relax is of course {\bf (H)}. 
Actually, it seems quite plausible that one can produce a stabilized 
version of the ``complete''\footnote{\emph{I.e.} dealing with abnormal 
sequences as well.} subresultant pseudo-remainder sequence algorithm 
following the same strategy. Nevertheless, this extension is not 
completely straightforward because designing it requires to understand 
precisely how the coefficients $c_i$'s (appearing in 
Eq.~\ref{eq:subressequence}) alter the behaviour of the precision. We
therefore let it as an open question.

As it was presented, Algorithm~\ref{algo:stabsubres} only accepts inputs 
consisting of a pair of \emph{monic} polynomials having the same degree. 
It is actually not difficult to make it work with all couples of 
polynomials $(A,B)$ such that $\lc(B)$ is invertible in $\A$ and $\deg A 
\geq \deg B$. Indeed, it is enough for this to replace line 2 by:
$$R_{d-1} \leftarrow (-1)^{\deg A - \deg B} (A \,\%\, B).$$
However, writing an extension of Algorithm~\ref{algo:stabsubres} that 
accepts all inputs seems much more tricky and this is the second open
question we would like to point out.

Beyond this, one may wonder if one can use similar technics to compute 
not only subresultants but cofactors as well. For those indexes $j$ such 
that $r_j$ is invertible in $\A$, the same analysis applies almost 
\emph{verbatim}. However for other indexes $j$, the differential 
computation seems to be much more subtle. One can get around this issue 
by using lifting technics only when $r_j$ is a unit in $\A$ and tracking 
precision naively otherwise: it is possible to get this way a stable 
algorithm whose average running time is acceptable but which seems to be 
bad in the worst case. Can we do better?

Another quite interesting question is those of designing an algorithm 
which combines the precision technology developed in this paper with the 
``half-\textsc{gcd}'' methods. It is actually closely related to the 
previous question because ``half-\textsc{gcd}'' methods make an 
intensive use of cofactors in order to speed up the computation.

\section{Conclusion: towards $p$-adic floats}

When computing with real numbers, computers very often use floating 
point arithmetic. 
The rough idea of this model consists in representating all real numbers 
using the same number of digits (the so-called \emph{precision}) and to
apply rounding heuristics when final digits are unsettled.
In comparison with arithmetic interval, floating point arithmetic has
two main advantages. First, it allows simple and fast implementations.
Second, experiments show that the obtained results have generally
more much correct digits that those predicted by arithmetic interval.
The counterpart is that, expect on small examples, obtaining proved
results is generally intractable.

In the $p$-adic setting, the analogue of floating point arithmetic has 
not been developed yet. One reason for this is probably the well-known 
saying: ``in the $p$-adic world, rounding errors do not accumulate''. 
Consequently one might expect that interval arithmetic would provide 
sharp results. Nonetheless this hope is failing and examples are basic 
and numerous: $p$-adic differential equations \cite{bostan, padicdiff}, 
LU factorization \cite{LU}, SOMOS 4 sequence \cite{padicprec}, 
resultants (this paper), \emph{etc.} Consequently, interval arithmetic 
is not as good as one might have expected at first. Therefore, it 
probably makes sense to seriously study the analogue of floating point 
arithmetic in a ultrametric context.

Let us describe quickly what might be this analogue and what are its
advantages and disadvantages. We keep the notations of the previous
sections: the letter $\A$ denotes a complete discrete valuation ring
with uniformizer $\pi$ and $K$ is its fraction field.
In the model of ultrametric floating point arithmetic, we fix a
positive integer $N$ (the \emph{precision}) and represent elements 
of $K$ by approximations of the form:
\begin{equation}
\label{eq:padicfloat}
\pi^e \cdot \sum_{i=0}^{N-1} x_i \pi^i
\end{equation}
where $e$ is a relative integer and the $x_i$'s are elements of a fixed 
set of representatives of $\A$ modulo $\pi$ with the convention that the 
representative of $0 \in k$ is $0 \in \A$.
We further assume that $x_0 \neq 0$, \emph{i.e.} $e$ is the valuation
of the sum \eqref{eq:padicfloat}. We see that this framework is quite
similar to usual floating point arithmetics: the integer $e$ plays the
role of exponent, the uniformizer $\pi$ plays the role of the basis
and the value $\sum_{i=0}^{N-1} x_i \pi^i$ plays the role of the
significand (the mantissa). It remains to define operations $\oplus$
and $\odot$ on approximations modeling addition and multiplication
on $K$ respectively. We do this as follows: given $x$ and $y$ two
elements of $K$ of the form Eq.~\eqref{eq:padicfloat}, we compute
$x+y$ (resp. $xy$) in $K$, expand it as a convergent series 
$\sum_{i=v}^{\infty} s_i \pi^i$ (with $s_v \neq 0$) and define $x 
\oplus y$ (resp. $x \odot y$) by truncating the series at $i = v+N$.

Similarly to real floating point arithmetic, the main advantages of 
ultrametric floating point arithmetic are the simplicity and the 
efficiency while the counterpart is the difficulty to get proved results. 
Moreover, the aforementioned examples are evidences that ultrametric 
floating point arithmetic may often compute much more correct digits 
than those predicted by an analysis based on interval arithmetic. In 
order to illustrate this last assertion, let us go back to the case of 
resultants discussed earlier in this paper. Let $A$ and $B$ be two monic 
polynomials of degree $d$ (picked at random) whose coefficients are all 
known at precision $O(\pi^N)$. We have proved that if we are using the
model of interval arithmetic, then the subresultant pseudo-remainder
sequence algorithm will output $\Res(A,B)$ at precision 
$O(\pi^{N-N_{\text{int}}})$ where $N_{\text{int}}$ grows linearly 
with respect to $d$ on average. On the other hand, if we are using 
ultrametric floating point arithmetic, then the same algorithm will
output $\Res(A,B)$ at precision $O(\pi^{N-N_{\text{float}}})$ where 
$N_{\text{float}}$ grows linearly with respect to $\log d$ on average.
We emphasize furthermore that this result is \emph{proved}! From this
point of view, floating point arithmetics seems to behave better in
the ultrametric setting: we may hope to get proved results relatively
cheaply.


\begin{thebibliography}{99}
\bibitem{andrews}
  G.~Andrews,
  \emph{The Theory of Partitions},
  Cambridge University Press (1976)
\bibitem{real}
  S.~Basu, R.~Pollack, M.-F.~Roy
  \emph{Algorithms in Real Algebraic Geometry},
  Springer-Verlag (2008), second edition
\bibitem{pari}
  C.~Batut, K.~Belabas, D.~Benardi, H.~Cohen, M.~Olivier, 
  \emph{User’s guide to PARI-GP} (1985--2013)
\bibitem{bostan}
  A.~Bostan, L.~Gonz\'alez-Vega, H.~Perdry, É.~Schost, 
  \emph{From Newton sums to coefficients: complexity issues in characteristic $p$}, 
  MEGA’05 (2005)
\bibitem{magma}
  W.~Bosma, J.~Cannon, C.~Payoust, 
  \emph{The Magma algebra system. I. The user language.}
  J. Symbolic Comput. {\bf 24} (1997), 235--265
\bibitem{padicprec}
  X.~Caruso, D.~Roe, T.~Vaccon,
  \emph{Tracking $p$-adic precision},
  LMS J. Comp. and Math. {\bf 17}, 274--294
\bibitem{LU}
  X.~Caruso,
  \emph{Random matrices over a DVR and LU factorization},
  to appear at J. Symb. Comp.
\bibitem{cohen}
  H.~Cohen,
  \emph{A course in Computational Algebraic Number Theory},
  Springer (1996)
\bibitem{kedlaya}
  K.~Kedlaya,
  \emph{Counting points on hyperelliptic curves using Monsky--Washnitzer cohomology}, 
  J. Ramanujan Math. Soc. {\bf 16} (2001), 323--338
\bibitem{padicdiff}
  P.~Lairez, T.~Vaccon,
  \emph{Computation of power series solutions with $p$-adic coefficients of certain differential equations},
  preprint (2014)
\bibitem{sage}
  W.~Stein et al.
  \emph{Sage Mathematics Software}, 
  The Sage Development Team (2005--2013)
\bibitem{winkler}
  F.~Winkler,
  \emph{Polynomial Algorithms in Computer Algebra},
  Springer Wien New Work (1996)
\end{thebibliography}
\end{document}